\documentclass[12pt]{article}
\usepackage{amsmath}
\usepackage{amssymb}
\usepackage{amsthm}
\usepackage[mathscr]{eucal}
\usepackage{latexsym}
\usepackage{amscd}
\usepackage{epic}
\usepackage{epsfig}

\begin{document}
\baselineskip=15pt

\voffset -1truecm
\oddsidemargin .5truecm
\evensidemargin .5truecm

\theoremstyle{plain}
\swapnumbers\newtheorem{lema}{Lemma}[section]
\newtheorem{prop}[lema]{Proposition}
\newtheorem{coro}[lema]{Corollary}
\newtheorem{teor}[lema]{Theorem}
\newtheorem{demo}[lema]{Proof of Theorem~\ref{teor:formula}}
\newtheorem{ejem}[lema]{Example}
\newtheorem{ejems}[lema]{Examples}
\newtheorem{eval}[lema]{Evaluations}
\newtheorem{obse}[lema]{Remark}
\newtheorem{obses}[lema]{Remarks}
\newtheorem{defi}[lema]{Definition}
\newtheorem{prob}[lema]{Open problem}
\newtheorem{probs}[lema]{Open problems}
\newtheorem{conje}[lema]{Conjecture}
\newtheorem{comen}[lema]{Comentary}
\newtheorem{producto}[lema]{Products}
\newtheorem{clasediag}[lema]{Diagram classes}
\newtheorem{young}[lema]{Principal border strips}
\newtheorem{removible}[lema]{Removable diagrams}
\newtheorem{collage}[lema]{Collages of skew diagrams}
\newtheorem{sorted}[lema]{Sorted decompositions}
\newtheorem{nota}[lema]{Notation}
\newtheorem{zele}[lema]{Pictures}
\newtheorem{evaliri}[lema]{Some formulas for {\mathversion{bold}$\multiliric{(n-d, \ol\nu)}$}}
\newtheorem{evakron}[lema]{Some formulas for {\mathversion{bold}${\sf g}(\la,\la, (n-d,\ol\nu))$}}
\newtheorem{polifun}[lema]{The main interval of {\mathversion{bold}$s_{\ol\nu}$}}
\newtheorem{mapfd}[lema]{The map {\mathversion{bold}$f_D$}}

\renewcommand{\refname}{\large\bf References}
\renewcommand{\thefootnote}{\fnsymbol{footnote}}

\def\natural{{\mathbb N}}
\def\noneg{\natural_0}
\def\entero{{\mathbb Z}}
\def\racional{{\mathbb Q}}
\def\real{{\mathbb R}}
\def\complejo{{\mathbb C}}
\def\planod{\natural \times \natural}
\def\vacio{\varnothing}

\def\vector#1#2{({#1}_1,\dots,{#1}_{#2})}
\def\flecha{\longrightarrow}
\def\asocia{\longmapsto}
\def\sii{\Longleftrightarrow}
\def\clase#1{\boldsymbol{[}\,{#1}\,\boldsymbol{]}}
\def\wt#1{\widetilde{#1}}
\def\ol#1{\overline{#1}}
\def\fun#1#2#3{{#1}\,\colon {#2} \flecha {#3}}
\def\nume#1{\boldsymbol{[}\,{#1}\,\boldsymbol{]}}
\def\negra#1{\boldsymbol{#1}}
\def\contenido{\subseteq}
\def\bili#1#2{\langle {#1}, {#2} \rangle}
\def\bilipunto{\bili{\,\cdot\,}{\cdot\,}}

\def\la{\lambda}
\def\longi#1{\ell({#1})}
\def\prof#1{{\sf d}({#1})}
\def\lintm{\lambda\cap\mu}
\def\rotap#1{{#1}^\circ}
\def\domina{\geqslant}
\def\dominada{\leqslant}
\def\liri{Littlewood-Richardson\ }
\def\lirim{Littlewood-Richardson multitableau\ }
\def\lirimx{Littlewood-Richardson multitableaux\ }
\def\coefliri#1#2#3{c^{\,#1}_{{#2}\,{#3}}}
\def\coefliril#1#2{c^{\,\lambda}_{{#1}\,{#2}}}
\def\ro1r{(\rho(1), \dots ,\rho(r))}
\def\rosec{\rho(1)\vdash \pi_1, \dots , \rho(r)\vdash \pi_r}
\def\clirir#1{c^{#1}_{\ro1r}}
\def\kos#1#2{K_{{#1}{#2}}}
\def\kosinv#1#2{K_{{#1}\,{#2}}^{(-1)}}
\def\multiliri#1{{\sf lr}(\la, \mu; {#1})}
\def\multiliric#1{{\sf lr}(\la, \la; {#1})}
\def\alt#1{{\sf ht}({#1})}
\def\sec#1#2{\varnothing = {#1}(0) \subset {#1}(1) \subset \cdots
\subset {#1}({#2}) }
\def\sesgala{\la/\alpha}
\def\sesgalb{\la/\beta}
\def\sesgamb{\mu/\beta}
\def\conte#1{\gamma({#1})}

\def\subcla{\trianglelefteq}
\def\subcladif{\vartriangleleft}
\def\supercla{\trianglerighteq}
\def\superclasdif{\vartriangleright}
\def\menori{\preccurlyeq}
\def\mayori{\succcurlyeq}
\def\menorie{\prec}
\def\mayorie{\succ}
\def\Diagra#1{{\sf R}_{\la}({#1})}
\def\diagra#1{{\sf r}_{\la}\left({#1}\right)}
\def\Diagrados#1#2{{\sf R}_{\la,\mu}\left({#1},{#2}\right)}
\def\diagrados#1#2{{\sf r}_{\la,\mu}\left({#1},{#2}\right)}
\def\Diagralibre#1#2{{\sf R}_{#1}\!\left({#2}\right)}
\def\diagralibre#1#2{{\sf r}_{#1}\!\left({#2}\right)}
\def\Interd#1{{\sf C}_{\la}(D_1, \dots, D_m;{#1})}
\def\interd#1{{\sf c}_{\la}(D_1, \dots, D_m;{#1})}
\def\Inter#1#2{{\sf C}_{#1}(D_1, \dots, D_m;{#2})}
\def\inter#1#2{{\sf c}_{#1}(D_1, \dots, D_m;{#2})}
\def\collaged#1{{\sf c}(D_1, \dots, D_m;{#1})}

\def\cara#1{\chi^{#1}}
\def\permu#1{\phi^{#1}}
\def\coefi{{\sf g}(\la,\mu,\nu)}
\def\ccuad#1{{\sf g}(\la,\la,{#1})}
\def\kron{\chi^\lambda\otimes\chi^\mu}
\def\tirabo#1{{\sf SBST}({#1})}
\def\signo#1{{\sf sgn}({#1})}

\def\gruposim#1{{\sf S}_{#1}}
\def\sime#1{{\sf S}_{#1}}

\def\coeflirim#1#2{c^{\,\mu}_{{#1}\,{#2}}}
\def\partil#1{\lambda({#1})}
\def\partim#1{\mu({#1})}
\def\partilintm#1{\lambda({#1})\cap\mu({#1})}
\def\pareja{(\lambda,\,\mu)}
\def\eseefe#1#2{{\sf {\mathscr{#1}}}({#2})}
\def\serifde#1#2{{\sf {#1}}({#2})}
\def\calmate#1#2{{\mathscr {#1}}_{#2}}
\def\retratos#1#2{{\mathscr P}({#1},{#2})}

\def\tablau{\tabla{ \ \\}}
\def\tabladh{\tabla{ \ & \ \\}}
\def\tabladv{\tabla{ \ \\ \ \\}}
\def\tablath{\tabla{ \ & \ & \ \\ }}
\def\tablatv{\tabla{ \ \\ \ \\ \ \\}}
\def\tablatdu{\tabla{ \ & \ \\ \ \\}}
\def\tablatud{\tabla{ & \ \\ \ & \ \\ }}
\def\tabladu{\raisebox{.1ex}{\tablau}\sqcup \raisebox{.1ex}{\tablau}}
\def\tablatu{\tabladu\sqcup \raisebox{.1ex}{\tablau}}
\def\tablatudh{\tablau\sqcup\tabladh}
\def\tablatudv{\tablau\sqcup\raisebox{.65ex}{\tabladv}}
\def\tablacu{\tablatu\sqcup \raisebox{.1ex}{\tablau}}
\def\tablacdudh{\tabladu \sqcup \tabladh}
\def\tablacdhdh{\tabladh\sqcup\tabladh}
\def\tablacdhdv{\tabladh\sqcup\raisebox{.65ex}{\tabladv}}
\def\tablacvc{\tabla{\ & \ & \ & \ \\}}
\def\tablacvuuuu{\raisebox{1.95ex}{\tabla{\ \\ \  \\ \ \\ \ \\}}}
\def\tablacvtu{\raisebox{.65ex}{\tabla{\ & \ & \ \\ \ \\}}}
\def\tablacvduu{\raisebox{1.3ex}{\tabla{\ & \ \\ \ \\ \ \\}}}
\def\tablacvdd{\raisebox{.65ex}{\tabla{\ & \ \\ \ & \ \\}}}
\def\tablacsut{\raisebox{.65ex}{\tabla{ & & \ \\ \ & \ & \ \\}}}
\def\tablacsuud{\raisebox{1.3ex}{\tabla{ & \ \\ & \ \\ \ & \ \\}}}
\def\tablacsedd{\raisebox{.65ex}{\tabla{ & \ & \ \\ \ & \ \\}}}
\def\tablacseudu{\raisebox{1.3ex}{\tabla{ & \ \\ \ & \ \\ \ \\}}}


\setlength\unitlength{0.08em}
\savebox0{\rule[-2\unitlength]{0pt}{10\unitlength}%
\begin{picture}(10,10)(0,2)
\put(0,0){\line(0,1){10}}
\put(0,10){\line(1,0){10}}
\put(10,0){\line(0,1){10}}
\put(0,0){\line(1,0){10}}
\end{picture}}

\newlength\cellsize \setlength\cellsize{18\unitlength}
\savebox2{%
\begin{picture}(18,18)
\put(0,0){\line(1,0){18}}
\put(0,0){\line(0,1){18}}
\put(18,0){\line(0,1){18}}
\put(0,18){\line(1,0){18}}
\end{picture}}
\newcommand\cellify[1]{\def\thearg{#1}\def\nothing{}%
\ifx\thearg\nothing
\vrule width0pt height\cellsize depth0pt\else
\hbox to 0pt{\usebox2\hss}\fi%
\vbox to 18\unitlength{
\vss
\hbox to 18\unitlength{\hss$#1$\hss}
\vss}}
\newcommand\tableau[1]{\vtop{\let\\=\cr
\setlength\baselineskip{-16000pt}
\setlength\lineskiplimit{16000pt}
\setlength\lineskip{0pt}
\halign{&\cellify{##}\cr#1\crcr}}}
\savebox3{%
\begin{picture}(15,15)
\put(0,0){\line(1,0){15}}
\put(0,0){\line(0,1){15}}
\put(15,0){\line(0,1){15}}
\put(0,15){\line(1,0){15}}
\end{picture}}
\newcommand\expath[1]{%
\hbox to 0pt{\usebox3\hss}%
\vbox to 15\unitlength{
\vss
\hbox to 15\unitlength{\hss$#1$\hss}
\vss}}
\newlength\celulita \setlength\celulita{5\unitlength}
\savebox3{%
\begin{picture}(5,5)
\put(0,0){\line(1,0){5}}
\put(0,0){\line(0,1){5}}
\put(5,0){\line(0,1){5}}
\put(0,5){\line(1,0){5}}
\end{picture}}
\newcommand\celificar[1]{\def\thearg{#1}\def\nothing{}%
\ifx\thearg\nothing
\vrule width0pt height\celulita depth0pt\else
\hbox to 0pt{\usebox3\hss}\fi%
\vbox to 5\unitlength{
\vss
\hbox to 5\unitlength{\hss$#1$\hss}
\vss}}
\newcommand\tablita[1]{\vtop{\let\\=\cr
\setlength\baselineskip{-16000pt}
\setlength\lineskiplimit{16000pt}
\setlength\lineskip{0pt}
\halign{&\celificar{##}\cr#1\crcr}}}

\newlength\cuadro \setlength\cuadro{7\unitlength}
\savebox4{%
\begin{picture}(7,7)
\put(0,0){\line(1,0){7}}
\put(0,0){\line(0,1){7}}
\put(7,0){\line(0,1){7}}
\put(0,7){\line(1,0){7}}
\end{picture}}
\newcommand\cuadrificar[1]{\def\thearg{#1}\def\nothing{}%
\ifx\thearg\nothing
\vrule width0pt height\cuadro depth0pt\else
\hbox to 0pt{\usebox4\hss}\fi%
\vbox to 7\unitlength{
\vss
\hbox to 7\unitlength{\hss$#1$\hss}
\vss}}
\newcommand\tabla[1]{\vtop{\let\\=\cr
\setlength\baselineskip{-16000pt}
\setlength\lineskiplimit{16000pt}
\setlength\lineskip{0pt}
\halign{&\cuadrificar{##}\cr#1\crcr}}}


\begin{centering}
{\Large\bf A diagrammatic approach to Kronecker squares}\\[1cm]
{\Large\sf Ernesto Vallejo\footnotemark[1]}\\[.1cm]
Centro de Ciencias Matem\'aticas\\
Universidad Nacional Aut\'onoma de M\'exico\\
Apartado Postal 61-3, Xangari\\
58089 Morelia, Mich., MEXICO\\
e-mail: {\tt vallejo@matmor.unam.mx}\\[.4cm]
\end{centering}

\vskip 2pc
\begin{abstract}
In this paper we improve a method of Robinson and Taulbee for computing Kronecker
coefficients and show that for any partition
$\ol\nu$ of $d$ there is a polynomial $k_{\ol\nu}$ with rational coefficients
in variables $x_C$, where $C$ runs over the set of isomorphism classes of connected
skew diagrams of size at most $d$, such that for all partitions $\la$ of $n$, the
Kronecker coefficient ${\sf g}(\la,\la, (n-d,\ol\nu))$ is obtained from $k_{\ol\nu}(x_C)$
substituting each $x_C$ by the number of partitions $\alpha$ contained in $\la$
such that $\la/\alpha$ is in the class $C$.
Some results of our method extend to arbitrary Kronecker coefficients.
We present two applications.
The first is a contribution to the Saxl conjecture, which asserts that if
$\rho_k = (k,k-1, \dots, 2, 1)$ is the staircase partition, then the Kronecker
square $\cara\rho \otimes \cara\rho$ contains every irreducible character of the
symmetric group as a component.
Here we prove that for any partition $\ol\nu$ of size $d$ there is a piecewise
polynomial function $s_{\ol\nu}$ in one real variable such that for all $k$
one has ${\sf g}(\rho_k, \rho_k, (|\rho_k| - d,\ol\nu)) = s_{\ol\nu}(k)$.
The second application is a proof of a new stability
property for arbitrary Kronecker coefficients.


\end{abstract}

\section{Introduction}

Let $\cara\la$ be the irreducible character of the symmetric group $\sime n$
associated to the partition $\la$ of $n$.
It is a major open problem in the representation theory of the symmetric group
in characteristic 0 to find a combinatorial or geometric
description of the multiplicity
\begin{equation} \label{ecua:kron}
 \coefi = \langle \kron, \cara\nu \rangle
\end{equation}
of $\cara\nu$ in the Kronecker product of $\kron$ of $\cara\la$
and $\cara\mu$ (here $\langle \cdot, \cdot \rangle$ denotes the
scalar product of complex characters).
Seventy five years ago Murnaghan~\cite{mur} published the first paper on
the subject.
Since then many people have searched out satisfactory ways for
computing the Kronecker coefficients $\coefi$.
Still, very little is known about the general problem.

Among the things known, there is a method for computing arbitrary Kronecker
coefficients.
It was introduced by Robinson and Taulbee in~\cite{rota} (see also~\cite[\S 3.4]{rob})
and reworked by Littlewood in~\cite{lit}.
In~\cite[\S 2.9]{jake} we can find the original method of Robinson and Taulbee,
another variation of it and some applications.
We will refer to this method and to any of its variations as the RT method.
Its main ingredients are the Jacobi-Trudi determinant, Frobenius
reciprocity and the Littlewood-Richardson rule.
Some of its applications can be found in~\cite{avala,sax,vpp,zis}.
Another variation of the RT method appears in~\cite[\S 6]{gare}.
Some applications of this variation are given in \cite{baor,baor2,rem}.

The version of the RT method from~\cite[p. 98]{jake} suggests how to systematize it
by means of the so called Littlewood-Richardson multitableaux (or simply LR multitableaux),
see~\cite{don,vpp}.
This technique, as it is already apparent from~\cite{jake}, is not only useful for
computations:
it also lead in~\cite{vejc} to a combinatorial proof of a stability property for Kronecker
coefficients observed by Murnaghan in~\cite{mur} and to the determination
of a lower bound for stability.
Other approaches to stability have been developed in \cite{bor,bri,lit2,thi}.
A dual approach of LR multitableaux was used in~\cite{vjac}
to study minimal components, in the dominance order of partitions, of Kronecker
products.
LR multitableaux were also used in~\cite{avva} to construct a one-to-one correspondence
between the set 3-dimensional matrices with integer entries and given 1-marginals
and the set of certain triples of tableaux.
This correspondence generalize the RSK correspondence and was used to describe
combinatorially some Kronecker coefficients.

In~\cite{vpp} we gave graphical formulas for the coefficients $\ccuad \nu$ of Kronecker
squares for all partitions $\nu = (n-d, \ol\nu)$ of depth $d\le 3$.
These computations were extended in~\cite{avala} to all partitions $\nu$ of depth $4$.
We include them in Section~\ref{sec:evalk} for completeness.
Some of them had appeared before in an algebraic but equivalent form in~\cite{gara,sax,zis};
some have already been applied in \cite{bes, bekl, bevw, ppv};
others may be suitable for future applications.

In this paper we prove that the formulas obtained in~\cite{vpp} and~\cite{avala}
are part of a general phenomenon (Theorem~\ref{teor:polikrontirabo}).
Namely, for each partition $\ol\nu$  of size $d$, there is a polynomial $k_{\ol\nu}(x_C)$
with rational coefficients in variables $x_C$, where $C$ runs over the set of isomorphism
classes of connected skew diagrams of size $|C| \le d$, such that
for each partition $\la$ of $n$, the Kronecker coefficient $\ccuad {(n-d,\ol\nu)}$
is obtained from $k_{\ol\nu}(x_C)$ by evaluating each $x_C$ at the number of partitions
$\alpha$ of $n-d$ contained in $\la$ such that $\la/\alpha$ is in $C$.
These polynomials do not depend on $\la$ or $n$.
In fact, we also show (Theorem~\ref{teor:tiras-borde-kron}) that $k_{\ol\nu}(x_C)$ can be modified
to obtain another polynomial $\wt k_{\ol\nu}(t_B)$ in variables $t_B$, where $B$ runs over
the set of isomorphism classes of connected border strips of size $|B| \le d$.
Then a similar evaluation of $\wt k_{\ol\nu}(t_B)$ also yields the corresponding
Kronecker coefficient.
Theorem~\ref{teor:polikrontirabo} is derived from Theorem~\ref{teor:kron-combin}
which is an enhancement of the RT method described above that gives a closed
combinatorial formula (up to signs) of Kronecker coefficients.
It incorporates the new notion of $\la$-removable diagram and a convenient use
of special border strip tableaux.
We will show its utility in Sections~\ref{sec:saxl} and~\ref{sec:estab}.
Theorem~\ref{teor:kron-combin} can be extended to arbitrary Kronecker coefficients
(see Theorem~\ref{teor:kron-combin-dos}).
This approach to Kronecker coefficients should be contrasted with
Murnaghan's~\cite{mur,mur2,mur3,mur4}, where
for any two partitions $\la = (n-a, \ol\la)$, $\mu= (n-b,\ol\mu)$ of $n$
a method is given to compute the expansion $\cara\la \otimes \cara\mu$ in terms
of $\ol\la$ and $\ol\mu$.
Another method for computing the same expansion is given by Littlewood in~\cite{lit2},
and a formula of Thibon that encompasses Murnaghan and Littlewood's approaches
appear in~\cite[\S 2]{thi}.

Besides we present two applications of the diagrammatic approach developed here.
The first is a contribution to the solution of the Saxl conjecture studied for the first time
in~\cite{ppv}.
Let $\rho_k$ be the staircase partition $(k, k-1, \dots, 2,1)$.
Saxl's conjecture asserts that the Kronecker product $\cara{\rho_k} \otimes \cara{\rho_k}$
contains every irreducible representation of the symmetric group as a component.
Here we show what we believe to be a surprising result: Theorem~\ref{teor:kronescalerapoli},
which says that for each partition $\ol\nu$
of size $d$ there is a piecewise polynomial function with rational coefficients
$\fun {s_{\ol\nu} }{[0,\infty)}\real$ such that
\begin{equation*}
{\sf g}(\rho_k, \rho_k, (|\rho_k| - d,\ol\nu)) = s_{\ol\nu}(k)
\end{equation*}
for all $k$ such that $(n_k-d,\ol\nu)$ is a partition.
This is the more surprising since the product $\cara{\rho_k} \otimes \cara{\rho_k}$ seems
to be the most difficult product of size $n_k$ to evaluate (see~\cite[p. 93]{lit}).
A further analysis (Theorem~\ref{teor:mi-saxl}) shows that
${\sf g}(\rho_k, \rho_k, (n_k-d,\ol\nu))$ is positive for all but at most $2d$
values of $k$.
The second result (Theorem~\ref{teor:estab-nuevo}) shows a new stability property
of arbitrary Kronecker coefficients that is evident once we know
Theorems~\ref{teor:kron-combin} and~\ref{teor:kron-combin-dos}.
Some of the results presented in this paper appeared previously
in preprint form in~\cite{vpp,vsqkron}.

In the last years it has been discovered that Kronecker coefficients
are related in an important way to two areas beyond algebraic combinatorics and
representation theory.
First there is the realization that Kronecker coefficients play an important role
in geometric complexity theory~\cite{muso1,muso2,blmw}.
Secondly, there is the discovery that Kronecker coefficients are related to the
quantum marginal problem~\cite{kly, chmi}.
The techniques developed here could be useful in solving problems regarding Kronecker
coefficients coming from these fields.

The paper is organized as follows.
In Section~\ref{sec:parti-tableaux} we review some known results needed in this paper
on the combinatorics of Young tableaux.
Section~\ref{sec:cara} contains some basic results about the character theory of
the symmetric group that will be used throughout.
In Section~\ref{sec:multita} we present several results concerning LR multitableaux.
Theorem~\ref{teor:coefi-dec} has not been published before, but it appears already in a
similar form in~\cite[Corollary~4.3]{vpp}.
In Section~\ref{sec:equiv} we introduce the notion of $\la$-removable diagram
(Paragraph~\ref{def:removible}).
This is the fundamental concept for our diagrammatic method.
A skew diagram $\sigma$ is $\la$-removable if there is a partition $\alpha$ contained in
$\la$ such that $\la/\alpha = \sigma$.
Given the isomorphism class $D$ of a skew diagram $\sigma$ and a partition $\la$, we denote by
$\diagra D$ the number of $\la$-removable diagrams isomorphic to $\sigma$.
In Theorem~\ref{teor:elbueno} we show that for each isomorphism class $D$ of skew diagrams,
the number $\diagra D$ can be expressed as a polynomial with rational coefficients
in variables $\diagra C$, where $C$ runs over the set of all isomorphism classes of
connected skew diagrams of size $|C| \le |D|$.
Section~\ref{sec:evallr} contains several calculations for the number of pairs
of LR multitableaux that will be used in the rest of the paper.
Section~\ref{sec:evalk} is the core of the paper.
It contains Theorems~\ref{teor:kron-combin}, \ref{teor:polikrontirabo}
and~\ref{teor:tiras-borde-kron} already mentioned.
We also include two formulas for $\ccuad {(n-d, \ol\nu)}$ when $\la$ is a rectangle
of size $n$ and $\ol\nu$ is either $(d)$ or $(1^d)$.
In Section~\ref{sec:saxl} we present our contribution to the Saxl conjecture, a table
with the polynomials $s_{\ol\nu}$ for all $|\ol\nu| \le 5$ and some conjectures.
In Section~\ref{sec:ext} we sketch how to extend the diagrammatic method to
arbitrary Kronecker coefficients.
Finally, Section~\ref{sec:estab} contains a new stability property that holds for
arbitrary Kronecker coefficients.

\section{Partitions and tableaux} \label{sec:parti-tableaux}

We assume the reader is familiar with the standard results in the
combinatorics of Young tableaux (see for example~\cite{ful, mac, sag, stan}).
In this section we review some of those basic results, definitions and notation
used in this paper.

We will use the following notation:
$\natural$ is the set of positive integers, $\noneg =\natural \cup \{0\}$, and,
for any $n\in \noneg$,  $\nume n = \{1, \dots, n\}$, so that $\nume 0 = \varnothing$.
If $\la$ is a partition, we denote its \emph{size} by $|\la |$
and its \emph{length} by $\longi{\la}$.
If $|\la|=n$, we also write $\la\vdash n$.
The {\em depth} of $\la$ is $\prof\la = |\la| - \la_1$.
For any composition $\pi= \vector \pi r$, that is, a vector of positive integers,
denote $\ol\pi =(\pi_2, \dots, \pi_r)$,
$|\pi| = \pi_1 + \cdots + \pi_r$ and $\longi \pi = r$.
If $|\pi|=n$, we also write $\pi \vDash n$.
Thus, for a partition $\la$ one has $\prof \la = |\ol\la|$.
Given two partitions $\la$, $\mu$ of $n$ we write $\la \domina \mu$
to indicate that $\la$ is greater than or equal to $\mu$ in the dominance
order of partitions.
The \emph{diagram} of a partition $\la = \vector \la p$, also denoted by $\la$,
is the set
\[
\la = \{\, ( i, j) \in \planod \mid i \in \nume p,\ j \in \nume{\la_i} \,\}.
\]
The identification of $\la$ with its diagram permits us to use
set theoretic notation for partitions.
The partition $\la^\prime$ \emph{conjugate} to $\la$ is obtained
by transposing the diagram of $\la$, that is,
$\la^\prime = \{\, (i, j) \mid (j,i) \in \la \,\}$.
If $\alpha$ is another partition and $\alpha \subseteq \la$,
we denote by $\lambda/ \alpha$ the \emph{skew diagram} consisting of the
pairs in $\la$ that are not in $\alpha$.
The \emph{size} $|\sesgala|$ of $\sesgala$ is its cardinality and
the \emph{conjugate} of $\sesgala$ is $\la^\prime/\alpha^\prime$.

The number of semistandard Young tableaux of shape $\la$ and
content $\pi$ is denoted by $\kos\la\pi$.
Let us arrange the partitions of $n$ in reverse lexicographic order
and form the matrix $K_n =(\kos\la\mu )$.
It is invertible over the integers.
Let $K_n^{-1}= (\kos\la\mu^{(-1)})$ denote its inverse.
We now explain the combinatorial description of the numbers
$\kos\la\mu^{(-1)}$ given in \cite{egre}.
Recall that a \emph{border strip} is a connected skew diagram
that contains no subset of the form $\raisebox{.15ex}{\tablacvdd}\,$.
A \emph{special border strip tableau} $T$ of shape $\mu$ is an
increasing sequence of partitions
\[
\varnothing =\mu(0)\subset \mu(1)\subset \cdots \subset \mu(c)=\mu,
\]
such that each $\mu(j)/\mu(j-1)$ is a special border strip of $\mu$, namely
a border strip that intersects the first column of $\mu$.
The \emph{sign} of $T$ is
\[
\signo T =  \prod_j (-1)^{(\alt{ \mu(j)/\mu(j-1)})} ,
\]
where the \emph{height} $\alt{ \mu(j)/\mu(j-1)}$ is the number of rows
of $\mu(j)/\mu(j-1)$ minus one.
The \emph{content}, denoted by $\conte T$, is the sequence
\[
\vert \mu(1)/\mu(0)\vert, \dots, \vert \mu(c)/\mu(c-1)\vert
\]
of sizes of the special border strips arranged in decreasing order.
Later it will be convenient to work with a reordering of $\conte T$.
Denote by $\tirabo \mu$ the set of all special border strip tableaux of
shape $\mu$.
E\~gecio\~glu and Remmel showed the following

\begin{teor}[\cite{egre}] \label{teor:er}
For any pair of partitions $\la$, $\mu$ of $n$ one has
\[
\kos\la\mu^{(-1)} =\sum_{T \in \tirabo \mu,\ \conte T = \la} \signo T.
\]
\end{teor}

For example, there are six special border strip tableaux of shape
$\mu=(5, 3, 2)$, which are indicated below with its content and sign.

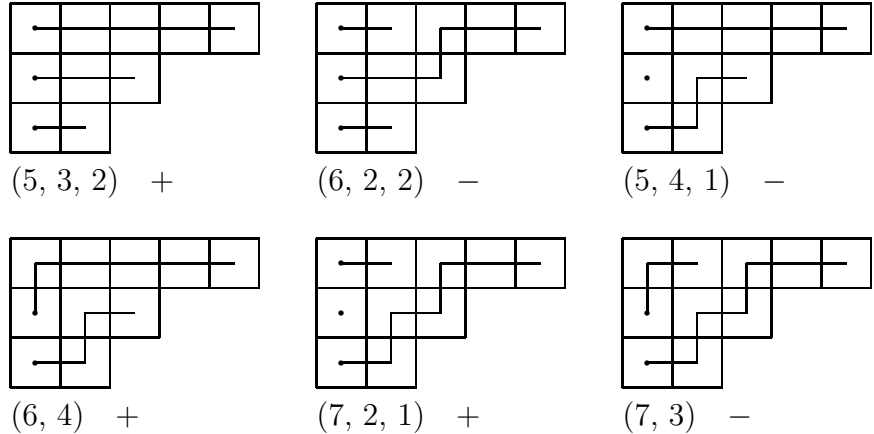
\[
\begin{array}{ccc}
\begin{picture}(100,80)(0,-80)

\drawline (0,0)  (100,0)
\drawline (0,-20) (100,-20)
\drawline (0,-40) (60,-40)
\drawline (0,-60) (40,-60)

\drawline (0,0)   (0,-60)
\drawline (20,0)  (20,-60)
\drawline (40,0)  (40,-60)
\drawline (60,0)  (60,-40)
\drawline (80,0)  (80,-20)
\drawline (100,0) (100,-20)

\put (10,-10){\circle*{2}}
\put (10,-30){\circle*{2}}
\put (10,-50){\circle*{2}}

\drawline (10,-10) (90,-10)
\drawline (10,-30) (50,-30)
\drawline (10,-50) (30,-50)

\put (0,-75){(5, 3, 2)\quad +}
\end{picture}
&
\quad
\begin{picture}(100,80)(0,-80)

\drawline (0,0)  (100,0)
\drawline (0,-20) (100,-20)
\drawline (0,-40) (60,-40)
\drawline (0,-60) (40,-60)

\drawline (0,0)   (0,-60)
\drawline (20,0)  (20,-60)
\drawline (40,0)  (40,-60)
\drawline (60,0)  (60,-40)
\drawline (80,0)  (80,-20)
\drawline (100,0) (100,-20)

\put (10,-10){\circle*{2}}
\put (10,-30){\circle*{2}}
\put (10,-50){\circle*{2}}

\drawline (10,-10) (30,-10)
\drawline (10,-30) (50,-30) (50,-10) (90,-10)
\drawline (10,-50) (30,-50)

\put (0,-75){(6, 2, 2)\quad $-$}
\end{picture}
\quad
&
\begin{picture}(100,80)(0,-80)

\drawline (0,0)  (100,0)
\drawline (0,-20) (100,-20)
\drawline (0,-40) (60,-40)
\drawline (0,-60) (40,-60)

\drawline (0,0)   (0,-60)
\drawline (20,0)  (20,-60)
\drawline (40,0)  (40,-60)
\drawline (60,0)  (60,-40)
\drawline (80,0)  (80,-20)
\drawline (100,0) (100,-20)

\put (10,-10){\circle*{2}}
\put (10,-30){\circle*{2}}
\put (10,-50){\circle*{2}}

\drawline (10,-10) (90,-10)
\drawline (10,-50) (30,-50) (30,-30) (50,-30)

\put (0,-75){(5, 4, 1)\quad $-$}
\end{picture}
\\
\begin{picture}(100,90)(0,-80)

\drawline (0,0)  (100,0)
\drawline (0,-20) (100,-20)
\drawline (0,-40) (60,-40)
\drawline (0,-60) (40,-60)

\drawline (0,0)   (0,-60)
\drawline (20,0)  (20,-60)
\drawline (40,0)  (40,-60)
\drawline (60,0)  (60,-40)
\drawline (80,0)  (80,-20)
\drawline (100,0) (100,-20)

\put (10,-30){\circle*{2}}
\put (10,-50){\circle*{2}}

\drawline (10,-30) (10,-10) (90,-10)
\drawline (10,-50) (30,-50) (30,-30) (50,-30)

\put (0,-75){(6, 4)\quad +}
\end{picture}
&
\begin{picture}(100,90)(0,-80)

\drawline (0,0)  (100,0)
\drawline (0,-20) (100,-20)
\drawline (0,-40) (60,-40)
\drawline (0,-60) (40,-60)

\drawline (0,0)   (0,-60)
\drawline (20,0)  (20,-60)
\drawline (40,0)  (40,-60)
\drawline (60,0)  (60,-40)
\drawline (80,0)  (80,-20)
\drawline (100,0) (100,-20)

\put (10,-10){\circle*{2}}
\put (10,-30){\circle*{2}}
\put (10,-50){\circle*{2}}

\drawline (10,-10) (30,-10)
\drawline (10,-50) (30,-50) (30,-30) (50,-30) (50,-10) (90,-10)

\put (0,-75){(7, 2, 1)\quad +}
\end{picture}
&
\begin{picture}(100,90)(0,-80)

\drawline (0,0)  (100,0)
\drawline (0,-20) (100,-20)
\drawline (0,-40) (60,-40)
\drawline (0,-60) (40,-60)

\drawline (0,0)   (0,-60)
\drawline (20,0)  (20,-60)
\drawline (40,0)  (40,-60)
\drawline (60,0)  (60,-40)
\drawline (80,0)  (80,-20)
\drawline (100,0) (100,-20)

\put (10,-30){\circle*{2}}
\put (10,-50){\circle*{2}}

\drawline (10,-30) (10,-10) (30,-10)
\drawline (10,-50) (30,-50) (30,-30) (50,-30) (50,-10) (90,-10)

\put (0,-75){(7, 3)\quad $-$}
\end{picture}
\end{array}
\]

The number of Littlewood-Richardson tableaux (or simply LR tableaux) of shape
$\lambda/\alpha$ and content $\mu$ will be denoted by $c^{\la/\alpha}_\mu$.
More generally, a sequence $T=(T_1, \dots , T_r)$ of tableaux is called an
\emph{LR multitableau} of \emph{shape} $\la/\alpha$ if
there is a sequence of partitions
\[
\alpha=\la(0) \subset \la(1) \subset \cdots \subset \la(r) = \la\, ,
\]
such that $T_i$ is LR tableau of shape $\la(i)/\la(i-1)$,
for all $i\in \nume r$.
The \emph{content} of $T$ is the sequence $\ro1r$ , where
$\rho(i)$ is the content of $T_i$;
the \emph{type} of $T$ is the composition $\vector{\pi}r$, where
$\pi_i = |\rho(i)|$.
For example,

\[
\raisebox{.6ex}{
\tableau{
{\bf 1} & {\bf 1} & {\bf 1} & {\bf 1} & {\it 1} & {\it 1} & {\it 1} & {\tt 1}
& {\tt 1} & {\tt 1} \\
{\bf 2} & {\bf 2} & {\bf 2} & {\bf 2} & {\it 2} & {\tt 2} & {\tt 2}& {\tt 2} \\
{\bf 3} & {\bf 3} & {\it 2}  & {\it   2}  & {\it 3} \\
{\it 3} & {\tt 3} \\}}
\]

\medskip
\noindent is an LR multitableau of shape $(10,8,5,2)$, content $((4,4,2),(3,3,2),(3,3,1))$
and type $(10,8,7)$.
The number of LR multitableaux of shape $\la/\alpha$ and content $\ro1r$
will be denoted by $\clirir {\la/\alpha}$.
Since there is a unique LR tableau of shape $\alpha$ and
content $\alpha$ we have the identity
\begin{equation} \label{ecua:lirisesgado}
\clirir{\la/\alpha} = c^{\la}_{(\alpha, \rho(1), \dots, \rho(r))}\, .
\end{equation}

Also note that
\[
\clirir{\la/\alpha} =
\sum_{\alpha=\la(0) \subset \la(1) \subset \cdots \subset \la(r) = \la}
\ \prod_{i=1}^r c^{\la(i)/\la(i-1)}_{\rho(i)} \, .
\]
Therefore, since $c^{\la/\alpha}_\beta =
c^{\la^\prime/\alpha^\prime}_{\beta^\prime}$,
we also have
\begin{equation} \label{ecua:conjuliri}
\clirir{\la/\alpha} =
c^{\la^\prime/\alpha^\prime}_{(\rho(1)^\prime, \dots, \rho(r)^\prime)}.
\end{equation}

\section{Characters of the symmetric group} \label{sec:cara}

We assume the reader is familiar with the standard results in the
representation theory of the symmetric group (see for example~\cite{jake, mac, sag, stan}).
In this section we review some of those basic results, definitions and notation
used in this paper.

For any partition $\la \vdash n$, we denote by $\cara\la$ the irreducible character
of $\sime n$ associated to $\la$, and, for any composition $\pi = \vector \pi r$ of $n$,
by $\permu\pi = {\rm Ind}_{\sime \pi}^{\sime n}( 1_{\pi})$ the
permutation character associated to $\pi$.
That is, $\permu\pi$ is the character induced from the trivial
character of the Young subgroup
$\sime \pi = \sime{\pi_1} \times \cdots \times \sime{\pi_r}$ to $\sime n$.
Note that if a composition $\rho$ of $n$ is a reordering of $\pi$, then
\begin{equation} \label{ecua:reorder}
\permu\pi = \permu\rho.
\end{equation}
The sets $\{ \cara\la \mid \la \vdash n \}$ and $\{ \permu\la \mid \la \vdash n \}$
are additive basis of the character ring of $\sime n$.
Both basis are related by Young's rule
\begin{equation*}
\permu\nu = \sum_{\la\vdash n} \kos\la\nu \cara\la. 
\end{equation*}
By the Jacobi-Trudi determinant one has
\begin{equation} \label{ec:jt}
\cara\nu= \sum_{\la \vdash n} \kos\la\nu^{(-1)} \permu\la.
\end{equation}
Then, Theorem~\ref{teor:er} implies
\begin{equation} \label{ec:jtdos}
\cara\nu= \sum_{T \in \tirabo \nu} \signo T \permu{\conte T}.
\end{equation}

\begin{defi} \label{defi:nuevo-cont}
{\em
Let $\nu = \vector \nu r$ be a partition and $T\in \tirabo\nu$.
Denote by $\tau(T)$ the vector $\vector tl$ such that $t_1$ is the size
of the border strip that contains the square $(1,\nu_1)$ and $t_2, \dots, t_l$
are the sizes of the remaining special border strips arranged in
non-increasing order.
The vector $\tau(T)$ is not a partition in general, but for $\nu_1$ big enough
it is a partition.
It will be useful to work with $\tau(T)$ as the \emph{content} of $T$
instead of the usual content $\gamma(T)$ defined above,
which by definition is always a partition.
Let also denote $\ol\tau(T)= (t_2, \dots, t_l)$
and $e(T)=|\ol\tau(T)|$.
}
\end{defi}

Equation~\eqref{ec:jtdos} can be rewritten as
\begin{equation} \label{ecua:jttres}
\cara\nu = \sum_{T\in \tirabo \nu}  \signo T \permu{\tau(T)}.
\end{equation}

In the rest of the paper we use the following

\begin{nota} \label{nota:partis}
{\em
Let $\ol\nu = (\nu_2, \dots, \nu_r)$ be a partition of $d$.
For any $n\ge d + \nu_2$ let
\[
\ol\nu(n) = (n-d, \ol\nu) = (n-d, \nu_2, \dots, \nu_r).
\]
This is a partition of $n$.
For simplicity we denote $\wt \nu = \ol\nu(d + \nu_2) = (\nu_2, \nu_2, \dots, \nu_r)$.
}
\end{nota}

Then we have

\begin{lema} \label{lema:biyec-tiraborde}
Let $\ol\nu$ be a partition of $d$.
Then, for any $n\ge d + \nu_2$ there is a sign preserving bijection
\[
\fun {B_n} {\tirabo{\wt\nu}} {\tirabo{\ol\nu(n)}}
\]
such that for all $T\in \tirabo{\wt\nu}$ one has
$\tau(B_n(T))= \tau(T) + (n-d-\nu_2, 0, \dots, 0)$.
\end{lema}
\begin{proof}
Let $T\in \tirabo{\wt\nu}$.
To the border strip in $T$ of size $t_1$ that contains the square $(1,\nu_2)$
we add the squares $(1,\nu_2+1), \dots, (1, n -d)$ and form a new
border strip of size $t_1 + n -d - \nu_2$.
Then $B_n(T)$ is the tableau formed by this new border strip plus the
remaining border strips of $T$.
This defines the desired bijection.
\end{proof}

Then we get a stable version (with respect to the first row) of equation~\eqref{ecua:jttres}:

\begin{coro} \label{coro:cara-estab}
Let $\ol\nu$ be a partition of $d$.
Then for any $n \ge d+\nu_2$ we have
\[
\cara{\,\ol\nu(n)} = \sum_{T\in \tirabo {\wt\nu}}  \signo T \, \permu{\tau(B_n(T))}.
\]
\end{coro}

This stability is illustrated by the following

\begin{ejem}
{\em
Let $\ol\nu = (1,1)$.
Then the elements in $\tirabo{\wt\nu}$ are
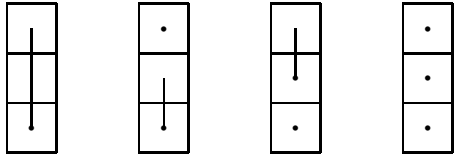
\[
\begin{array}{cccc}
\begin{picture}(30,70)(-5,-5)
\drawline (0,0)  (20,0)
\drawline (0,20) (20,20)
\drawline (0,40) (20,40)
\drawline (0,60) (20,60)

\drawline (0,0)   (0,60)
\drawline (20,0)  (20,60)

\put (10,10){\circle*{2}}
\drawline (10,10)  (10,50)
\end{picture}

&\quad

\begin{picture}(30,70)(-5,-5)
\drawline (0,0)  (20,0)
\drawline (0,20) (20,20)
\drawline (0,40) (20,40)
\drawline (0,60) (20,60)

\drawline (0,0)   (0,60)
\drawline (20,0)  (20,60)

\put (10,10){\circle*{2}}
\put (10,50){\circle*{2}}

\drawline (10,10)  (10,30)
\end{picture}

&\quad

\begin{picture}(30,70)(-5,-5)
\drawline (0,0)  (20,0)
\drawline (0,20) (20,20)
\drawline (0,40) (20,40)
\drawline (0,60) (20,60)

\drawline (0,0)   (0,60)
\drawline (20,0)  (20,60)

\put (10,10){\circle*{2}}
\put (10,30){\circle*{2}}

\drawline (10,30)  (10,50)
\end{picture}

&\quad

\begin{picture}(30,70)(-5,-5)
\drawline (0,0)  (20,0)
\drawline (0,20) (20,20)
\drawline (0,40) (20,40)
\drawline (0,60) (20,60)

\drawline (0,0)   (0,60)
\drawline (20,0)  (20,60)

\put (10,10){\circle*{2}}
\put (10,30){\circle*{2}}
\put (10,50){\circle*{2}}
\end{picture}
\end{array}
\]
and the corresponding elements in $\tirabo {\ol\nu(4)}$ under $B_n$ are
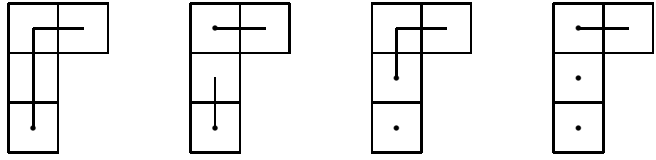
\[
\begin{array}{cccc}
\begin{picture}(50,70)(-5,-5)
\drawline (0,0)  (20,0)
\drawline (0,20) (20,20)
\drawline (0,40) (40,40)
\drawline (0,60) (40,60)

\drawline (0,0)   (0,60)
\drawline (20,0)  (20,60)
\drawline (40,40) (40,60)

\put (10,10){\circle*{2}}

\drawline (10,10) (10,50)
\drawline (10,50) (30,50)
\end{picture}

&\quad

\begin{picture}(50,70)(-5,-5)
\drawline (0,0)  (20,0)
\drawline (0,20) (20,20)
\drawline (0,40) (40,40)
\drawline (0,60) (40,60)

\drawline (0,0)   (0,60)
\drawline (20,0)  (20,60)
\drawline (40,40) (40,60)

\put (10,10){\circle*{2}}
\put (10,50){\circle*{2}}

\drawline (10,10)  (10,30)
\drawline (10,50) (30,50)
\end{picture}

&\quad

\begin{picture}(50,70)(-5,-5)
\drawline (0,0)  (20,0)
\drawline (0,20) (20,20)
\drawline (0,40) (40,40)
\drawline (0,60) (40,60)

\drawline (0,0)   (0,60)
\drawline (20,0)  (20,60)
\drawline (40,40) (40,60)

\put (10,10){\circle*{2}}
\put (10,30){\circle*{2}}

\drawline (10,30)  (10,50)
\drawline (10,50) (30,50)
\end{picture}

&\quad

\begin{picture}(50,70)(-5,-5)
\drawline (0,0)  (20,0)
\drawline (0,20) (20,20)
\drawline (0,40) (40,40)
\drawline (0,60) (40,60)

\drawline (0,0)   (0,60)
\drawline (20,0)  (20,60)
\drawline (40,40) (40,60)

\put (10,10){\circle*{2}}
\put (10,30){\circle*{2}}
\put (10,50){\circle*{2}}

\drawline (10,50) (30,50)
\end{picture}
\end{array}
\]
Thus
\begin{align*}
\cara{\wt\nu} = \cara{\nu(3)}
&= \permu{(3)} - \permu{(1,2)} - \permu{(2,1)} + \permu{(1,1,1)}
=  \permu{(3)} - 2\permu{(2,1)} + \permu{(1,1,1)}\\
\intertext{and}
\cara{\ol\nu(4)}
&=  \permu{(4)} - \permu{(2,2)} - \permu{(3,1)} + \permu{(2,1,1)}.
\end{align*}
}
\end{ejem}

We will need the following well-known result.
A proof of which can be found in~\cite{mac}.

\begin{lema} \label{lema:skew-standar}
Let $\alpha$, $\la$ be partitions such that $\alpha\subseteq \la$.
Then,

(1) the degree $\cara{\la/\alpha}(1)$ of the skew character
$\cara{\la/\alpha}$ is the number $f^{\la/\alpha}$ of
standard Young tableaux of shape $\la/\alpha$;

(2) $\cara{\la/\alpha} = \sum_{\mu \vdash |\la/\alpha|} c^{\la/\alpha}_\mu \cara\mu$.
\end{lema}

\begin{nota}
{\em
Let $\pi= \vector \pi r$ be a composition of $d$.
Denote the corresponding multinomial coefficient by
\[
\binom{d}{\pi} = \frac{d!}{\pi_1 ! \cdots \pi_r !}.
\]
}
\end{nota}

\begin{lema} \label{lema:kostka-multinom}
Let $\nu$ be a partition of $d$.
Then
\[
\sum_{\mu \vdash d} K^{(-1)}_{\mu,\nu} \binom{d}{\mu} = f^\nu.
\]
\end{lema}
\begin{proof}
Since $\permu \mu (1) = \binom{d}{\mu}$ and $\cara\nu (1) = f^\nu$
the claim follows from identity~\eqref{ec:jt}.
\end{proof}

A nice instance of the previous lemma is the following identity:

\begin{coro}
Let $d\in \natural$.
Then
\[
\sum_{\pi \vDash d} (-1)^{|\pi| - \longi \pi} \binom{d}{\pi} = 1.
\]
\end{coro}
\begin{proof}
Let $T\in \tirabo{(1^d)}$.
Order the special border strips $\zeta_1, \dots, \zeta_l$ of $T$ from
top to bottom and let $\pi_i$ be the length of $\zeta_i$.
Then $\pi(T)= \vector \pi l$ is a composition of $d$ and $\signo T = (-1)^{d - l}$.
The correspondence $T \asocia \pi(T)$ is a bijection between $\tirabo{(1^d)}$
and the set of compositions of $d$.
The claim follows now from Lemma~\ref{lema:kostka-multinom} when $\nu = (1^d)$.
\end{proof}

There are three kinds of products of characters we need to deal with.

\begin{producto}
{\em
Let $\varphi$, $\psi$ be characters of $\sime m$, $\sime n$, respectively.
Then

(1) $\varphi \times \psi$ denotes the character of $\sime m \times
\sime n$ given by $(\varphi \times \psi) (\sigma, \tau) =
\varphi(\sigma) \psi(\tau)$, for all $\sigma\in \sime m$,
$\tau \in \sime n$;

(2) $\varphi \bullet \psi = {\rm Ind}_{\sime m\times \sime n}
^{\sime{m+n}} (\varphi \times \psi)$; and

(3) if $m=n$, the Kronecker product of $\varphi$ with $\psi$, that is,
the character of $\sime m$ defined by
$(\varphi \otimes \psi) (\sigma) = \varphi(\sigma) \psi(\sigma)$, for all
$\sigma\in\sime m$.
}
\end{producto}

\begin{obse}
{\em
It is possible to state our results in the language of symmetric functions.
This can be done with the usual dictionary: $\cara\la$ corresponds to the
Schur function $s_\la$, the permutation character $\permu\la$ to the complete
homogeneous symmetric function $h_\la$, the product $\bullet$ to the usual product
of symmetric functions, the Kronecker product $\otimes$ to the inner product
of symmetric functions $*$, and the scalar product of characters $\bilipunto$
to the scalar product of symmetric functions.
}
\end{obse}

\section{Littlewood-Richardson multitableaux} \label{sec:multita}

In this section we introduce certain pairs of LR multitableaux which permit us
to work with the RT method in a more systematic way.

Let $\pi = \vector{\pi}r$ be a composition of $n$ and,
for each $i \in \nume r$, let $\rho(i)$ be a partition of $\pi_i$.
Then, by the Littlewood-Richardson rule we have
\[
\cara{\rho(1)} \bullet \cdots \bullet \cara{\rho(r)} =
\sum_{\la\vdash n} \clirir \la \cara\la.
\]
Therefore
$\clirir \la = \langle \cara\la ,
\cara{\rho(1)} \bullet \cdots \bullet \cara{\rho(r)} \rangle$
for all $\la\vdash n$.
More generally, if $\la$, $\alpha$ are partitions,
$\alpha \subseteq \la$ and $|\la/\alpha| =n$, we have by~\cite[2.4.16]{jake} or
by~\cite[I,5]{mac}.
\begin{align} \label{ecua:sesgaliri}
\clirir \sesgala =
c^\la_{(\alpha, \rho(1), \dots, \rho(r))} & =
\left<\cara\la , \cara\alpha \bullet
\cara{\rho(1)} \bullet \cdots \bullet \cara{\rho(r)} \right>
\notag \\
&= \left< \cara{\la/\alpha} ,
\cara{\rho(1)} \bullet \cdots \bullet \cara{\rho(r)} \right>.
\end{align}

\begin{defi} \label{defi:lr}
{\em
Let $\alpha$, $\beta$, $\la$, $\mu$ be partitions such that $\alpha \subseteq \la$,
$\beta \subseteq \mu$ and $|\la/\alpha| = |\mu/\beta| = n$, and let $\pi$ be a composition
of $n$.
Then we define
\begin{equation*}
{\sf lr}(\la/\alpha, \mu/\beta; \pi)
= \sum_{\rosec} \clirir{\la/\alpha} \clirir{\mu/\beta}.
\end{equation*}
This is the number of pairs $(S,T)$ of LR multitableaux of shape
$(\la/\alpha,\mu/\beta)$, same content and type $\pi$.
In other words, $S=\vector Sr$ is an LR multitableau of shape
$\la/\alpha$, $T=\vector Tr$ is an LR multitableau of shape
$\mu/\beta$ and both $S_i$ and $T_i$ have the same content $\rho(i)$ for
some partition $\rho(i)$ of $\pi_i$, for $i\in \nume r$.
}
\end{defi}

The following lemma is a consequence of Frobenius reciprocity and the
Littlewood-Richardson rule.

\begin{lema} \label{lema:lr}
Let $\la$, $\mu$, $\alpha$, $\beta$, $\pi$ be as above.
Then
\[
{\sf lr}(\la/\alpha, \mu/\beta; \pi)
= \langle \cara{\la/\alpha}\otimes \cara{\mu/\beta}, \permu\pi \rangle.
\]
\end{lema}

Then we have, by equation~\eqref{ecua:reorder}, the following

\begin{coro}
If $\rho$ is obtained by reordering the coordinates of $\pi$, then
\[
{\sf lr}(\la/\alpha, \mu/\beta; \rho) =
{\sf lr}(\la/\alpha, \mu/\beta; \pi)\, .
\]
\end{coro}

A combinatorial proof of this corollary follows from the tableau switching
properties studied in~\cite{bss}.

\begin{lema} \label{lema:itera}
Let $\la$, $\mu$ be partitions of $n$ and $\pi =\vector{\pi}r$ be a
composition of $n$.
Then
\[
\multiliri \pi = \sum_{\alpha \vdash \pi_1}
{\sf lr}(\la/\alpha, \mu/\alpha; \overline\pi).
\]
\end{lema}
\begin{proof}
Since, by definition,
\begin{equation*}
\multiliri \pi = \sum_{\rho(1)\vdash \pi_1, \dots, \rho(r) \vdash \pi_r}
\clirir {\la}\clirir {\mu}
\end{equation*}
the claim follows from \eqref{ecua:lirisesgado} letting $\alpha = \rho(1)$
and then again by Definiton~\ref{defi:lr}.
\end{proof}

\begin{coro}
Let $\la$, $\mu$ be partitions of $n$ and $\pi =\vector{\pi}r$ be a
composition of $n$.
If $|\la\cap\mu| < \pi_1$, then $\multiliri \pi = 0$.
\end{coro}

Now we put to use the numbers $\multiliri \pi$ and show that the Kronecker coefficient
${\sf g}(\la,\mu, (n-d, \ol\nu))$ can be computed, up to signs, in a purely
combinatorial manner, that depends only on $\la$, $\mu$ and $\ol\nu$.

\begin{prop} \label{prop:robtau}
Let $\ol\nu = (\nu_2, \dots, \nu_r)$ be a partition of $d$ and $n \ge d + \nu_2$.
Then, for any partitions $\la$, $\mu$ of $n$ we have
\begin{equation*}
{\sf g}(\la,\mu, (n-d,\ol\nu)) =
\sum_{T\in \tirabo{\wt\nu}} {\sf sign}(T)\, \multiliri {\tau(B_n(T))} .
\end{equation*}
\end{prop}
\begin{proof}
This follows from equation~\eqref{ecua:kron}, Corollary~\ref{coro:cara-estab}
and Lemma~\ref{lema:lr}.
\end{proof}

\begin{obse}
{\em
This formula gives a more systematic way of looking at the RT method.
Computationally there is no improvement.
However, from a theoretical point of view, this combinatorial
description of $\coefi$ is useful, as we will show below.
A similar result can be found in~\cite{don}.
Donin expresses there the coefficient $\coefi$ as the determinant of a matrix
with entries of the form $\multiliri \pi$ for some $\pi$'s.
In our case, the explicit use of special rim hook tableaux $T$ and the choice
of their contents $\tau(B_n(T))$ will have some advantages.
}
\end{obse}

Let $\la = \vector \la p \vdash n$,
let $a$, $b$ be integers such that $1\le a < b\le p$, and let the vector
$\mu = \vector \mu p$ be defined by
\begin{equation} \label{ec:cubre}
\mu_i=
\begin{cases}
\la_i,   & \text{if $ i\neq a,b$;} \\
\la_a+1, & \text{if $i=a$;} \\
\la_b-1, & \text{if $i=b$.}
\end{cases}
\end{equation}
Assume that $\mu$ is a partition.
Then $\mu \domina \la$.
Let $\la_{a,b}=(\la_a,\la_b)$ and let $\widehat\la_{a,b}$ be obtained from $\la$ by
deleting $\la_a$ and $\la_b$.
The following result is contained in the proof of the
main theorem of~\cite{livi}.

\begin{lema} \label{lema:lv}
Let $\la$ and $\mu$ be as above.
Then
\begin{equation*}
\permu\la - \permu{\mu} = \cara{\la_{a,b}} \bullet \permu{\widehat\la_{a,b}}.
\end{equation*}
\end{lema}
\begin{proof}
It follows from equation~\eqref{ec:jtdos} that
$\cara{\la_{a,b}} = \permu{\la_{a,b}} - \permu{(\mu_a, \mu_b)}$.
Therefore
\begin{equation*}
\cara{\la_{a,b}} \bullet \permu{\widehat\la_{a,b}} =
\permu{\la_{a,b}} \bullet \permu{\widehat\la_{a,b}} - \permu{(\mu_a, \mu_b)} \bullet \permu{\widehat\la_{a,b}}
= \permu\la - \permu{\mu}.
\end{equation*}
The lemma is proved.
\end{proof}

Let ${\mathscr P}_n$ denote the lattice of partitions of $n$ with the dominance order.

\begin{teor} \label{teor:dec}
Let $\chi$ be the character of a complex representation of ${\sf S}_n$.
Then the map $\fun {F_\chi} {{\mathscr P}_n} \natural$
defined by $F_\chi (\la)= \langle \chi, \permu\la \rangle$ is
weakly decreasing.
\end{teor}
\begin{proof}
It is enough to prove that if $\mu$ covers
$\la$ in ${\mathscr P}_n$, then $F_\chi(\la)\ge F_\chi(\mu)$.
But it is known~\cite[Thm. 1.4.10]{jake}, \cite{livi} that if $\mu$
covers $\la$, then $\mu$ satisfies~\eqref{ec:cubre}.
Therefore by Lemma~\ref{lema:lv} $F_\chi(\la)-F_\chi(\mu) =
\langle \chi, \cara{\la_{a,b}}\bullet \permu{\widehat\la_{a,b}} \rangle \ge0$.
\end{proof}
The following theorem has not been published before.
It appeared for the first time in Corollary~4.3 in~\cite{vpp} when
$\alpha$ and $\beta$ are the empty partition,
and in the present form in Corollary~4.9 in~\cite{vsqkron}.

\begin{teor} \label{teor:coefi-dec}
Let $\la$, $\mu$, $\alpha$, $\beta$ be partitions such that
$\alpha\subseteq\la$, $\beta\subseteq\mu$ and both
$\la/\alpha$ and $\mu/\beta$ have size $n$.
If $\rho$, $\sigma$ are partitions of $n$ and $\rho \dominada \sigma$,
then
\begin{equation*}
{\sf lr}(\la/\alpha, \mu/\beta;\rho) \ge {\sf lr}(\la/\alpha, \mu/\beta; \sigma).
\end{equation*}
\end{teor}
\begin{proof}
It follows from Theorem~\ref{teor:dec} and Lemma~\ref{lema:lr}.
\end{proof}

\section{Diagram classes} \label{sec:equiv}

In this section we introduce the notions of $\la$-removable diagrams and
collages of diagrams.
The first one will be useful in the computation of
multiplicities of components in Kronecker squares.
It is also very useful in the study of the Saxl conjecture (see Section~\ref{sec:saxl}).
The main result in this section (Theorem~\ref{teor:elbueno}) gives a recursive method to compute, for each
isomorphism class of skew diagrams $D$, a polynomial $p_D(x_C)$ with rational
coefficients in the variables $x_C$, where $C$ runs over the set of isomorphism classes
of connected skew diagrams $C$ of size at most $|D|$.
If $D$ is the class of a non-connected diagram, there is an evaluation of $p_D(x_C)$
which expresses, for all partitions $\la$, the number of $\la$-removable diagrams
in the class $D$ in terms of the numbers of $\la$-removable connected diagrams
of smaller size.
In some cases we compute this polynomials explicitly.

\begin{defi} \label{defi:diagramas}
{\em
Let $\rho$, $\sigma$ be skew diagrams.
Denote by $\rho = \cup_{i\in \nume m} \rho_i$, $\sigma = \cup_{j\in \nume n} \sigma_j$
their decompositions into connected components.
Then $\rho$ and $\sigma$ are called \emph{isomorphic diagrams} if $m=n$ and there is a
permutation $\pi \in \sime m$ such that, for each $i\in \nume m$, there is a bijective
translation $\fun{f_i}{\rho_i}{\sigma_{\pi(i)}}$.
The map $\fun f\rho\sigma$ defined, for each $x\in \rho_i$, by $f_i(x)$ is called
an \emph{isomorphism of diagrams}.
Clearly $f^{-1}$ is also an isomorphism of diagrams.
For example, if $\la =(3,2,1)$, $\alpha=(1,1)$, $\mu=(5,3,3,2)$ and $\beta=(4,3,1,1)$,
then $\sesgala$ is isomorphic to $\sesgamb$.
An isomorphism $\fun f \sesgala \sesgamb$ is defined by $f(i,j)= (i+2,j)$ for
$i=1$, $2$, and $f(3,1)=(1,5)$.
Note that the second coordinates of the squares of $\mu/\beta$ are either 2, 3 or 5.

\begin{center}
\begin{picture} (80,40)(-40,0)
\put (-40,17){$\sesgala =$}
\drawline (0,5) (10,5)
\drawline (0,15) (20,15)
\drawline (10,25) (30,25)
\drawline (10,35) (30,35)
\drawline (0,5) (0,15)
\drawline (10,5) (10,35)
\drawline (20,15) (20,35)
\drawline (30,25) (30,35)
\end{picture}
\quad
\begin{picture} (80,40)(-40,0)
\put (-40,17){$\mu/\beta =$}
\drawline (10,0) (20,0)
\drawline (10,10) (30,10)
\drawline (10,20) (30,20)
\drawline (40,30) (50,30)
\drawline (40,40) (50,40)
\drawline (10,0) (10,20)
\drawline (20,0) (20,20)
\drawline (30,10) (30,20)
\drawline (40,30) (40,40)
\drawline (50,30) (50,40)
\end{picture}
\end{center}
Therefore, if $\fun f \rho \sigma$ is an isomorphism of diagrams, then
$\rho$ is connected, respectively, a border strip if and only if $\sigma$ is
connected, respectively, a border strip.
}
\end{defi}

\begin{young} \label{para:young}
{\em
Consider the partial order defined on $\planod$ by
$(a,b) \menori (c,d)$ if and only if $a\le c$ and $b \le d$ and
denote by $\fun {p_k} {\planod} \natural$ the $k$-th projection, $k=1,2$.
Let $\sigma$ be a nonempty connected skew diagram and
let $a_\sigma = \min\{ p_1(x) \mid x \in \sigma \}$ and
$b_\sigma = \min\{ p_2(x) \mid x \in \sigma \}$.
Then $a_\sigma$ is the index of the highest row in $\sigma$ and $b_\sigma$
is the index of the leftmost column of $\sigma$.
The \emph{Young hull} of $\sigma$ is the skew diagram defined as
\begin{equation*}
y(\sigma) = \{ x \in \planod \mid x \mayori (a_\sigma,b_\sigma) \text{ and }
x \menori z \text{ for some } z\in \sigma \}.
\end{equation*}
Then $y(\sigma)$ is the translation of some partition diagram.
For example

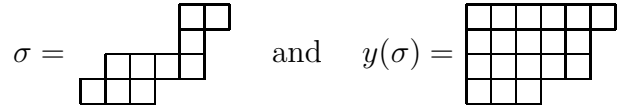
\begin{center}
\begin{picture} (87,40)(-40,0)
\put (-27,17){$\sigma =$}
\drawline (0,0) (30,0)
\drawline (0,10) (50,10)
\drawline (10,20) (50,20)
\drawline (40,30) (60,30)
\drawline (40,40) (60,40)
\drawline (0,0) (0,10)
\drawline (10,0) (10,20)
\drawline (20,0) (20,20)
\drawline (30,0) (30,20)
\drawline (40,10) (40,40)
\drawline (50,10) (50,40)
\drawline (60,30) (60,40)
\end{picture}
\begin{picture} (60,40)(0,0)
\put (26,17){and}
\end{picture}
\begin{picture} (102,40)(-40,0)
\put (-42,17){$y(\sigma) = $}
\drawline (0,0) (30,0)
\drawline (0,10) (50,10)
\drawline (0,20) (50,20)
\drawline (0,30) (60,30)
\drawline (0,40) (60,40)
\drawline (0,0) (0,40)
\drawline (10,0) (10,40)
\drawline (20,0) (20,40)
\drawline (30,0) (30,40)
\drawline (40,10) (40,40)
\drawline (50,10) (50,40)
\drawline (60,30) (60,40)
\end{picture}
\end{center}
For each $i\in \entero$, let $D_i(\sigma) =\{ (a,b) \in \sigma \mid b-a =i \}$ be the
$i$-th \emph{diagonal} of $\sigma$.
For each $i \in \entero$ such that $D_i(\sigma) \neq \varnothing$, let $x_i$ be the square
in $D_i(\sigma)$ whose coordinate sum is bigger than the coordinate sums of all other squares
in $D_i(\sigma)$.
Since $\sigma$ is connected, the set $b(\sigma)$ of all such $x_i$'s is a border strip.
We call $b(\sigma)$ the \emph{principal border strip} of $\sigma$.
Below we show with dots the principal border strip of a skew diagram.

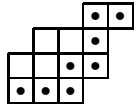
\begin{center}
\begin{picture} (50,40)(0,0)
\drawline (0,0) (30,0)
\drawline (0,10) (40,10)
\drawline (0,20) (40,20)
\drawline (10,30) (50,30)
\drawline (30,40) (50,40)
\drawline (0,0) (0,20)
\drawline (10,0) (10,30)
\drawline (20,0) (20,30)
\drawline (30,0) (30,40)
\drawline (40,10) (40,40)
\drawline (50,30) (50,40)
\put (5,5){\circle*{3}}
\put (15,5){\circle*{3}}
\put (25,5){\circle*{3}}
\put (25,15){\circle*{3}}
\put (35,15){\circle*{3}}
\put (35,25){\circle*{3}}
\put (35,35){\circle*{3}}
\put (45,35){\circle*{3}}
\end{picture}
\end{center}
We also define $y(\vacio) = \vacio$ and $b(\vacio) = \vacio$.
Clearly, if $\zeta$ is a border strip, then $b(y(\zeta)) =\zeta$, and if
$\sigma$ is a connected skew diagram such that $\zeta \subseteq \sigma \subseteq y(\zeta)$,
then $b(\sigma)= \zeta$ and $y(\sigma) = y(\zeta)$.
}
\end{young}

\begin{defi}
{\em
Let $\la = \vector \la p$, $\mu = \vector \mu q$, $\alpha = \vector \alpha a$ and $\beta = \vector \beta b$
be partitions such that $\alpha \subseteq \la$ and $\beta \subseteq \mu$.
If $a<p$ or $b < q$ we let $\alpha_{p+1}, \dots, \alpha_p$ and
$\beta_{b+1}, \cdots,  \beta_q $ be zero.
The \emph{disjoint union} $\sesgala \sqcup\mu/\beta$ of $\sesgala$ and $\mu/\beta$
is the skew diagram $\nu/\gamma$ where
\begin{equation*}
\nu = (\la_1 + \mu_1, \dots, \la_1 + \mu_q, \la_1, \dots, \la_p)
\text{ and }
\gamma = (\la_1 + \beta_1, \dots, \la_1 + \beta_q, \alpha_1, \dots, \alpha_p).
\end{equation*}
For example, if $\la = (3,2,1)$, $\alpha = (1^2)$, $\mu = (4,2)$ and $\beta = (1)$, then
the disjoint union $\sesgala \sqcup\mu/\beta$ is
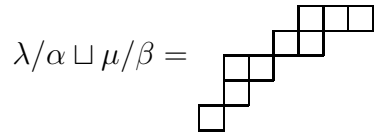
\begin{center}
\begin{picture} (145,50)(-75,0)
\put (-75,27){$ \sesgala \sqcup\mu/\beta = $}
\drawline (0,0) (10,0)
\drawline (0,10) (20,10)
\drawline (10,20) (30,20)
\drawline (10,30) (50,30)
\drawline (30,40) (70,40)
\drawline (40,50) (70,50)
%
\drawline (0,0) (0,10)
\drawline (10,0) (10,30)
\drawline (20,10) (20,30)
\drawline (30,20) (30,40)
\drawline (40,30) (40,50)
\drawline (50,30) (50,50)
\drawline (60,40) (60,50)
\drawline (70,40) (70,50)
\end{picture}
\end{center}
}
\end{defi}

We will work with isomorphism classes of diagrams.
For this we need the following definitions:

\begin{clasediag} \label{para:clase-diag}
{\em
If $\sigma$ is a skew diagram, its \emph{diagram class} $\clase \sigma$ is the set of
all skew diagrams isomorphic to $\sigma$.
Let $D = \clase \sesgala$.
The \emph{size} of $D$ is $|D| = |\sesgala|$ and the \emph{conjugate}
$D^{\,\prime}$ of $D$ is the diagram class $\clase{\la^{\,\prime} /\alpha^{\,\prime}}$.
If $E=\clase\sesgamb$ is another diagram class, denote
$D \sqcup  E = \clase {\sesgala \sqcup \sesgamb }$.
Note that $D \sqcup  E = E \sqcup D$.
We say that $D$ is \emph{connected} if $\sesgala$ is connected.
If $D=\clase \sigma$ is connected, then $B(D) = \clase {b(\sigma)}$ is the
\emph{principal border strip} of $D$.
A diagram class $C$ is a {\em subclass} of $D = \clase \sesgala$,
denoted by $C \subcla D$, if there is a partition $\gamma$
such that $\alpha \subseteq \gamma \subseteq \la$ and
$\clase{\la/\gamma} = C$.
For example, $\tabladh$\, is a subclass of $\,\raisebox{.65ex}{\tablatud}\,$,
but not of $\,\raisebox{.65ex}{\tablatdu}\,$.
}
\end{clasediag}

We introduce now the notion of $\la$-removable diagram, which will be fundamental
in the rest of the paper.

\begin{removible} \label{def:removible}
{\em
Let $\la$ be a partition.
A skew diagram $\rho$ is $\la$-\emph{removable} if there is a partition
$\alpha\subseteq \la$ such that $\rho = \la/\alpha$.
Define the set of $\la$-\emph{removable diagrams} in a diagram class $D$ by
\begin{equation*}
\Diagra D =\{\, \rho \subseteq \la \mid \rho
\text{ is $\la$-removable and }  \clase\rho = D \,\},
\end{equation*}
and let $\diagra D = \# \Diagra D$, be the number of elements in
$\Diagra D$.
For example, $\diagra \varnothing =1$ and $\diagra{\tablau}$ is the number
of removable squares of $\la$ (see~\cite[p.~202]{bekl}).
Removable squares are also called \emph{corners} of the diagram
(see, for example, ~\cite[p.~16]{pak}).

Conjugation of skew diagrams defines a bijection between the sets
$\Diagra{D^{\,\prime}}$ and ${\sf R}_{\la^\prime}(D)$, therefore
\begin{equation}\label{ecua:conjugada}
\diagra{D^{\,\prime}}={\sf r}_{\la^\prime}(D).
\end{equation}
}
\end{removible}

\begin{ejem}
{\em
Let $\la=(4,3,1)$, then $\diagra {\,\tablau\,} =3$,
$\diagra{\, \tablau \sqcup \tabladh\, }  =2$,
$\diagra{\,\raisebox{.65ex}{\tablatdu}\,} = 1$ and
$\diagra{\, \tablau \sqcup \raisebox{.65ex}{\tabladv}\,} = 0$.
}
\end{ejem}

The proofs of the following two lemmas are straightforward.

\begin{lema} \label{lema:redu-tirabo}
Let $\la$ be a partition and $C$ be a connected diagram class.
Then
\begin{equation*}
\diagra C = \diagra {B(C)}.
\end{equation*}
\end{lema}

\begin{lema} \label{lema:ui-remueve}
Let $\la$ be a partition.
If $\{\rho_i\}_{i\in \nume n}$ is a finite family of $\la$-removable diagrams, then
$\cup_{i\in \nume n} \rho_i$ and $\cap_{i\in \nume n} \rho_i$ are $\la$-removable
diagrams.
\end{lema}

For the construction of the polynomial $p_D(x_C)$ in Theorem~\ref{teor:elbueno} we need
the following definition:

\begin{collage} \label{parra:collage}
{\em
Let $\la$ be a partition and $\sigma$ be a $\la$-removable diagram.
A \emph{collage} of $\sigma$ is a sequence $\vector \rho m$ of $\la$-removable
diagrams such that $\sigma = \cup_{i\in \nume m} \rho_i$, where
the union is not necessarily disjoint.
For any finite list $D_1, \dots, D_m$ of diagram classes define
the set of \emph{collages} of $\sigma$ assembled from $D_1, \dots, D_m$ by
\begin{equation*}
\Interd \sigma = \{\, \vector \rho m \in \Diagra {D_1} \times
\cdots \times \Diagra {D_m} \mid \cup_{i\in \nume m} \rho_i = \sigma \,\}
\end{equation*}
and
\begin{equation*}
\interd \sigma =\# \Interd \sigma.
\end{equation*}
Let $\vector \rho m$ be a collage of $\sigma$.
Denote $D  = \clase\sigma$ and $D_i = \clase{\rho_i}$.
Then we say that $D$ is a \emph{collage} of $D_1, \dots, D_m$.
In this case we have that $D_i \subcla D$ for each $i$ and
\begin{equation*}
|D| \le |D_1|+ \cdots + |D_m|.
\end{equation*}
}
\end{collage}

\begin{ejem} \label{ejem:lema-falla}
{\em
Let $\la =(4,3,2,1)$, $\alpha =(3,2,1)$.
Let $D_1 = D_2 = \raisebox{.65ex}{\tabla{ & \ \\ \ &  \\ }} = \tabladu$
and $D = D_1 \sqcup D_2 = \clase{\la/\alpha}$.
Then $\Diagra {D_1} = \{ \la/\beta \mid \beta\in S \}$, where
\begin{equation*}
S = \{ (3,3,2), (3,2,2,1), (3,3,1,1), (4,3,1), (4,2,2), (4,2,1,1)\},
\end{equation*}
and ${\sf C}_\la(D_1, D_2; \sesgala)$ is the set of pairs of skew
shapes $\{ (\la/\alpha_1, \la/\alpha_2) \mid (\alpha_1, \alpha_2) \in T \}$
where $T$ is the set formed by the following pairs
\begin{gather*}
[(3,3,2), (4,2,1,1)], [(3,2,2,1),(4,3,1)], [(3,3,1,1),(4,2,2)], \\
[(4,2,1,1), (3,3,2)], [(4,3,1),(3,2,2,1)], [(4,2,2),(3,3,1,1)].
\end{gather*}
Therefore ${\sf c}_\la(D_1, D_2; \sesgala) = 6$ .
}
\end{ejem}

We next show that the definition of $\interd \sigma$ depends only on the isomorphism
class of $\sigma$.

\begin{lema} \label{lema:indep}
Let $\alpha$, $\beta$, $\la$, $\mu$ be partitions such that
$\alpha\subseteq \la$ and $\beta\subseteq \mu$.
If $\sesgala$ and $\sesgamb$ are isomorphic diagrams, then for any diagram classes
$D_1, \dots, D_m$ one has
\begin{equation*}
\inter \la \sesgala = \inter \mu \sesgamb.
\end{equation*}
\end{lema}
\begin{proof}
Let $\fun f \sesgala \sesgamb$ be an isomorphism.
We define a bijection
\begin{equation*}
\fun {\Phi_f} {\Inter \la \sesgala} {\Inter \mu \sesgamb}
\end{equation*}
by $\Phi_f\vector \rho m = (f(\rho_1), \dots, f(\rho_m))$.
We have to show that $f(\rho_i)\in {\sf R}_\mu(D_i)$, for each $i \in \nume m$.
Let $\alpha_i$ be a partition such that $\alpha_i \subseteq \la$ and $\la/\alpha_i = \rho_i$.
First we show that there is a partition  $\beta_i \subseteq \mu$ such that
$f(\rho_i) = \mu/\beta_i$.
Since $\rho_i \subseteq \sesgala$ we have that $\alpha\cap\rho_i = \vacio$
and $\alpha \contenido \alpha_i$.
Then, $\mu$ is the disjoint union $\beta \cup f(\alpha_i/\alpha) \cup f(\la/\alpha_i)$.
Let $\beta_i = \beta \cup f(\alpha_i/\alpha)$.
We have to show that $\beta_i$ is a partition.
Let $x$, $y\in \mu$ be such that $x \menori y$ (recall Paragraph~\ref{para:young}) and $y\in \beta_i$.
We claim $x\in \beta_i$.
If $x\notin \beta$, then $y \in f(\alpha_i/\alpha)$.
Since $x \menori y$, both elements are in a connected component $\mu/\delta$
of $\sesgamb$.
Since $\sesgala$ and $\sesgamb$ are isomorphic, there is a connected component
$\la/\gamma$ of $\sesgala$ such that the restriction
$\fun {f_|} {\la/\gamma} {\mu/\delta}$ is a translation.
Therefore $f^{-1}(x) \menori f^{-1}(y)$.
Thus $f^{-1}(x) \in \alpha_i$ and $x\in f(\alpha_i/\alpha) \subseteq \beta_i$.
We have proved that $\beta_i$ is a partition.
Therefore $f(\rho_i)$ is a $\mu$-removable diagram.
Since the restriction $\fun {f_|} {\rho_i} {f(\rho_i)}$ is an isomorphism of diagrams,
$f(\rho_i) \in {\sf R}_{\mu}(D_i)$.
The sequence $(f(\rho_i), \dots, f(\rho_m))$ is a collage of $\sesgamb$
because $f$ is a bijection.
Then $\Phi_f$ is a well defined map.
Its inverse is $\Phi_{f^{-1}}$.
This completes the proof of the lemma.
\end{proof}

By the previous lemma we can make the following

\begin{defi} \label{defi:inter-indep}
{\em
Let $D_1, \dots, D_m$, $D= \clase \sesgala$ be diagram classes.
Define
\begin{equation*}
\collaged D = \interd \sesgala.
\end{equation*}
This is the number of collages of $D$ assembled from $D_1, \dots, D_m$.
}
\end{defi}

\begin{ejem}
{\em
Let
\begin{equation*}
D_1 = \raisebox{1.3ex}{\tabla{ & \ & \  \\ \ & \ \\ \ \\}}\ , \qquad
D_2 = \raisebox{.65ex}\tablatdu \ , \qquad
D   = \raisebox{1.95ex}{\tabla{ & & \ & \ \\ & \ & \ \\ \ & \ \\ \ \\}}\ ,
\end{equation*}
then ${\sf c}( D_1, D_2; D) =2$ and ${\sf c}(D_2,D_2,D_2; D) = 6$.
}
\end{ejem}

\begin{sorted} \label{para:sorted}
{\em
Given a diagram class $E$ and a positive integer $n$ we denote by $E^{\,\sqcup \, n}$
the disjoint union $E \sqcup \cdots \, \sqcup E$ of $n$ copies of $E$.
Let $D$ be a nonempty diagram class and $D_1 \sqcup \cdots \sqcup D_m$ be a decomposition
of $D$ into its connected components.
A \emph{sorted decomposition} of $D$ is a reordering
$ C_1^{\,\sqcup\, a_1} \sqcup \cdots\, \sqcup C_k^{\,\sqcup \, a_k}$
of $D_1 \sqcup \cdots \sqcup D_m$ such that $C_1, \dots, C_m$ are pairwise distinct
classes.
Thus $a_1, \dots, a_k$ are positive and $a_1 + \cdots + a_k = m$.
}
\end{sorted}

\begin{lema} \label{lema:collage}
Let $D$ be a nonempty diagram class, $D_1 \sqcup \cdots \sqcup D_m$ be a decomposition
of $D$ into its connected components and
$C_1^{\,\sqcup\, a_1} \sqcup \cdots\, \sqcup C_k^{\,\sqcup \, a_k}$
be a sorted decomposition of $D$.
Then
\begin{equation*}
\collaged D = a_1 ! \cdots a_k!
\end{equation*}
\end{lema}
\begin{proof}
Let $D = \clase \sesgala$.
For each $i\in \nume k$ and $j\in \nume{a_i}$ let $\gamma_{i,j}$ be a $\la$-removable diagram such that
$\clase {\gamma_{i,j}} = C_i$ and
\begin{equation*}
\la/\alpha = \bigcup_{i \in \nume k} \bigcup_{j \in \nume {a_i}} \gamma_{i,j},
\end{equation*}
as a disjoint union.
The vector
\begin{equation*}
(\gamma_{1,1}, \dots, \gamma_{1,a_1}, \gamma_{2,1}, \dots, \gamma_{k,1},
\dots, \gamma_{k,a_k})
\end{equation*}
is in $\Interd \sesgala$.
This vector is divided in $k$ blocks of sizes $a_1, \dots, a_k$.
Any permutation of $\gamma_{i,j}$'s within a block yields a new element in $\Interd \sesgala$.
Therefore
\begin{equation*}
a_1 ! \cdots a_k! \le \interd \sesgala.
\end{equation*}
Let $ (\delta_{1,1}, \dots, \delta_{k,a_k})$ be an element in $\Interd \sesgala$.
Then $\clase{\delta_{i,j}} = C_i$ and
\begin{equation*}
\textstyle  \bigcup_{i \in \nume k} \bigcup_{j\in \nume {a_i}} \delta_{i,j} =
\sesgala.
\end{equation*}
Looking at the cardinalities of both sides one concludes that the union is disjoint.
Since the $C_i$'s are pairwise distinct and connected, for each $i\in \nume k$
there is a permutation $\sigma_i \in \sime {a_i}$ such that
$\delta_{i,j} = \gamma_{i, \sigma_i(j)}$.
Therefore $ a_1 ! \cdots a_k! \ge \interd {\la/\alpha}$.
The claim follows from Lemma~\ref{lema:indep} and Definition~\ref{defi:inter-indep}.
\end{proof}

Example~\ref{ejem:lema-falla} shows that the hypothesis of connectivity on the $C_i$'s cannot be
removed from Lemma~\ref{lema:collage}.

\begin{lema} \label{lema:producto}
Let $\la$ be a partition and let $D_1, \dots, D_m$ be diagram classes.
Then
\begin{equation*}
\diagra{D_1} \cdots \diagra{D_m} =
\sum_E \collaged E \, \diagra E,
\end{equation*}
where the sum runs over all diagram classes $E$ such that
$D_i \subcla E$, for each $i \in \nume m$, and $|E| \le |D_1|+ \cdots + |D_m|$.
\end{lema}
\begin{proof}
Note that
\begin{align*}
\Diagra{D_1} \times \cdots \times \Diagra{D_m} & =
\bigcup_{\alpha\subseteq\la} \Interd \sesgala \\
& = \bigcup_E \bigcup_{ \sigma \in \Diagra E} \Interd \sigma ,
\end{align*}
where $E$ runs over the set of diagram classes of $\la$-removable diagrams.
Since the unions are disjoint, the identity follows from Lemma~\ref{lema:indep}
and Definition~\ref{defi:inter-indep}.
If for some diagram class $E$ one has $\collaged E >0$, it follows from the
definition of collage (Paragraph~\ref{parra:collage}) that each $D_i$ is a subclass
of $E$ and that $|E| \le |D_1|+ \cdots + |D_m|$.
\end{proof}

The next proposition will be needed in the proof of Theorem~\ref{teor:elbueno}.
However, in some instances, the computation of $\diagra D$ is done more efficiently
by other means.
See for example Lemma~\ref{lema:remueve-snakes}.

\begin{prop} \label{prop:producto}
Let $\la$ be a partition, $D$ be a nonempty diagram class, $D_1 \sqcup \cdots \sqcup D_m$
be a decomposition of $D$ into its connected components and
$C_1^{\,\sqcup\, a_1} \sqcup \cdots\, \sqcup C_k^{\,\sqcup \, a_k}$
be a sorted decomposition of $D$.
Then
\begin{equation*}
\diagra D = \frac{1}{a_1! \cdots a_k !} \Biggl[
\diagra{C_1}^{a_1} \cdots \diagra{C_k}^{a_k} -
\sum_{\substack{ E\, \supercla C_1, \dots, C_k \\  |E| < |D|}}
\collaged E \,\diagra E \Biggr].
\end{equation*}
\end{prop}
\begin{proof}
By Lemma~\ref{lema:producto} we have
\begin{equation*}
\collaged D \,\diagra D  = \diagra{D_1} \cdots \diagra{D_m} -
\sum_{\substack{ E\, \supercla C_1, \dots, C_k \\  |E| < |D|}}
\collaged E \,\diagra E.
\end{equation*}
The claim follows now from Lemma~\ref{lema:collage}.
\end{proof}

\begin{nota} \label{nota:remueve-sesgado}
{\em
Let $E=\clase{\la/\alpha}$ be a diagram class.
We denote by $r_E$ the number of removable squares in $\la$ that are not in
$\alpha$.
This definition does not depend on the choice of the representative.
For example, if $\la =(4,3,2,2)$ and $\alpha=(3,3)$, then
$r_{\clase\la}=3$ and $r_{\clase \sesgala}=2$.
}
\end{nota}

\begin{lema} \label{lema:binom}
Let $\Delta_k= \raisebox{.1ex}{\tablau} \sqcup \cdots \sqcup \raisebox{.1ex}{\tablau}$ be
the disjoint union of $k$ squares, $k\ge 1$, and let $E$ be a (possibly empty) diagram class
with no connected component equal to $\tablau\,$.
Then
\begin{equation*}
\diagra{E\sqcup \Delta_k} = \diagra E \binom{\diagra{\tablau\,} - r_E}{k}.
\end{equation*}
In particular
\begin{equation*}
\diagra{\Delta_k} = \binom{\diagra{\tablau\,}}{k}.
\end{equation*}
\end{lema}
\begin{proof}
It is possible to give a direct proof.
We prefer, however, to show an application of Lemma~\ref{lema:producto}.
The proof is by induction on $k$.
Let $k=1$.
Since $E$ contains no component equal to $\tablau\,$, we have
${\sf c}(E,\Delta_1; E\sqcup \Delta_1) =1$ and
${\sf c}(E,\Delta_1; E) = r_E$.
Then, by Lemma~\ref{lema:producto},
\begin{equation*}
\diagra{E\sqcup \Delta_1} = \diagra E \diagra{\Delta_1} - r_E\, \diagra E.
\end{equation*}
Thus, the formula holds for $k=1$.
Suppose now, by induction hypothesis, that the formula holds for $k\ge 1$.
We have
\begin{equation*}
{\sf c}(E \sqcup \Delta_k,\Delta_1; E\sqcup \Delta_{k+1}) =k+1
\quad {\rm and}\quad {\sf c}(E\sqcup \Delta_k,\Delta_1; E\sqcup \Delta_k) = r_E +k.
\end{equation*}
Therefore, by Lemma~\ref{lema:producto}, we have
\begin{align*}
\diagra{E\sqcup \Delta_{k+1}} & = \frac{1}{k+1}
\left[ \diagra{E\sqcup \Delta_k} \diagra{\Delta_1}
-(r_E +k)\, \diagra{ E\sqcup \Delta_k}\right] \\
&= \frac{\diagra{E\sqcup \Delta_k}}{k+1}
\left[ \diagra{\,\tablau\,} - r_E -k\right].
\end{align*}
The claim now follows from the induction hypothesis for
$\diagra{E\sqcup \Delta_k}$.
\end{proof}

\begin{teor}\label{teor:elbueno}
For any nonempty diagram class $D$ there is a polynomial $p_D(x_C)$ with rational
coefficients, in the variables $x_C$, where $C$ runs over the set of connected diagram
classes of size $1 \le |C| \le |D|$, such that for all partitions $\la$ the number
$\diagra D$ is obtained from $p_D(x_C)$ evaluating each $x_C$ at $\diagra C$.
So, we have
\begin{equation*}
\diagra D = p_D(\diagra C).
\end{equation*}
If $D$ is not connected, the polynomial $p_D(x_C)$ depends only on
the variables $x_C$ with $|C| < |D|$.
\end{teor}
\begin{proof}
The proof is by induction on the size of $D$.
If $|D|= 1$, then $D= \tablau \,$ and the polynomial
$p_D(x_C)= x_D$ satisfies the theorem.
Now, suppose $|D|=n>1$ and that, by induction hypothesis, the statement of the theorem holds
for all diagrams of size smaller than $n$.
If $D$ is connected, the polynomial $p_D(x_C)= x_D$ satisfies the theorem.
If $D$ is not connected, let $D_1 \sqcup \cdots \sqcup D_m$ be a decomposition of $D$ into its
connected components and $C_1^{\,\sqcup\, a_1} \sqcup \cdots\, \sqcup C_k^{\,\sqcup \, a_k}$
be a sorted decomposition of $D$.
Then by Proposition~\ref{prop:producto} we have that
\begin{equation*}
\diagra D = \frac{1}{a_1! \cdots a_k !}\Biggl[
\diagra{C_1}^{a_1} \cdots \diagra{C_k}^{a_k} -
\sum_{\substack{ E\, \supercla C_1, \dots, C_k \\  |E| < |D|}}
\collaged E \,\diagra E \Biggr].
\end{equation*}
For any diagram $E$ in the sum above, let $p_E(x_C)$ the polynomial
that exists by induction hypothesis.
Define
\begin{equation*}
p_D(x_C)= \frac{1}{a_1! \cdots a_k !}\Biggl[
(x_{C_1})^{a_1} \cdots (x_{C_k})^{a_k} -
\sum_{\substack{ E\, \supercla C_1, \dots, C_k \\  |E| < |D|}}
\collaged E \, p_E(x_C) \Biggr].
\end{equation*}
Then we have that $\diagra D = p_D( \diagra C)$.
The last claim is clear.
\end{proof}

\begin{obse} \label{obse:polis}
{\em
Note that sometimes, as we did in the proof, depending on the context,
we allow the polynomial $p_D(x_C)$ to be defined
in a polynomial ring with more variables, namely, with variables $x_C$
running over connected diagram classes $C$ of size $|C| \le d$ for some $d > |D|$.
Of course the coefficients of the new variables will be zero.
In the same spirit we define $p_\vacio(x_C) = 1$ and the number of its variables
will be determined by the context.
}
\end{obse}

\begin{obse} \label{obse:tiras-borde}
{\em
Let $D$ be a diagram class.
If, in the polynomial $p_D(x_C)$, we substitute each variable $x_C$
by a new variable $t_{B(C)}$, where $B(C)$ is the principal border strip
of $C$, we obtain a new polynomial $\wt p_D(t_B)$ with rational coefficients
in (possibly fewer) variables $t_B$, where $B$ runs over the set of
border strip classes $B$ of size $1 \le |B| \le |D|$.
Because of Lemma~\ref{lema:redu-tirabo} and Theorem~\ref{teor:elbueno},
this polynomial satisfies the identity
\begin{equation*}
\diagra D = \wt p_D( \diagra B).
\end{equation*}
}
\end{obse}

\begin{ejems} \label{ejem:reduc}
{\em
In the following examples (needed in Sections~\ref{sec:evallr} and~\ref{sec:evalk})
we write some evaluations $\diagra D = p_D(\diagra C)$ of the polynomials just defined.
For these examples we use Lemma~\ref{lema:binom} and Proposition~\ref{prop:producto}.
\begin{enumerate}

\item
$\diagra{\,\tablatudh\,} = \left[ \diagra{\,\tablau\,} -1\right]\,
\diagra{\,\tabladh\,}$.

\item
$\diagra{\,\tablacdudh\,} =
\dbinom{\diagra{\,\tablau\,} -1}{2}\,\diagra{\,\tabladh\,}$.

\item
$\diagra{\,\tablacdhdh\,} = \dbinom{\diagra{\,\tabladh\,}}{2}$.

\item
$\diagra{\, \tablacdhdv \,} = \diagra{\,\tabladh\,}\,
\diagra{\, {\raisebox{.65ex}{\tabladv}}\,}
- \diagra{\, {\raisebox{.65ex}{\tablatud}}\, }$.

\item
$\diagra{\,\tablau\sqcup\tablath\,} =
[\diagra{\,\tablau\,} -1]\, \diagra{\,\tablath\,}$.

\item
$\diagra{\,\tablau\sqcup{\raisebox{.65ex}{\tablatdu}}\,} =
[ \diagra{\,\tablau\,} -2]\,
\diagra{\,{\raisebox{.65ex}{\tablatdu}}\,}$.

\item
$\diagra{\,\tablau\sqcup{\raisebox{.65ex}{\tablatud}}\,} =
[ \diagra{\,\tablau\,} -1]\,
\diagra{\,{\raisebox{.65ex}{\tablatud}}\,}$.
\end{enumerate}
}
\end{ejems}

Note the values of $\diagra{\ast}$ for the conjugated skew diagrams of
the examples above can be computed from identity~\eqref{ecua:conjugada}
in this section.

\section{The numbers {\mathversion{bold}$\multiliric \pi$}}
\label{sec:evallr}

In this section we prove some properties of the numbers ${\sf lr}(\sigma,\sigma;\pi)$
for a skew diagram $\sigma$.
The main result of this section (Theorem~\ref{teor:poliliritirabo}) shows, that for each composition $\pi$,
there is a polynomial $q_{\ol\pi}(x_C)$ (depending only on $\ol\pi$) with rational coefficients in
variables $x_C$, where $C$ runs over the set of connected diagram classes of size $1 \le |C| \le |\ol\pi|$,
such that, for all partitions $\la$, the evaluation of $q_{\ol\pi}(x_C)$ at the numbers $\diagra C$ yields
the number $\multiliric \pi$.
We will use these polynomials in the next section for the computation of Kronecker squares.
At the end of the section we compute explicitly $\multiliric \pi$ for all partitions
$\pi$ of depth at most 4.

\begin{lema} \label{lema:mismoscaras}
Let $\rho$, $\sigma$ be isomorphic skew diagrams.
Then $\cara\rho = \cara\sigma$.
\end{lema}
\begin{proof}
We use the Murnaghan-Nakayama formula for skew characters~\cite[7.17.3]{stan}.
Let $\fun f \rho \sigma$ be an isomorphism of skew diagrams.
If $\fun T \sigma {\nume n}$ is a border strip tableau of shape $\sigma$ and type $\gamma$,
then, it follows from Definition~\ref{defi:diagramas} that
$\fun {T \circ f} \rho {\nume n}$ is a border strip tableau of shape
$\rho$ and type $\gamma$ and that both tableaux have the same height.
Thus the Murnaghan-Nakayama formula implies our claim.
\end{proof}

\begin{coro} \label{coro:indep}
Let $\rho$, $\sigma$ be isomorphic skew diagrams.
Then

(1) for any composition $\pi$ of $|\sigma|$ we have
${\sf lr}(\rho,\rho;\pi) = {\sf lr}(\sigma,\sigma;\pi)$;

(2) for all partitions $\rho(1), \dots, \rho(r)$ we have
$\clirir{\rho} =\clirir{\sigma}$;

(3) $f^\rho = f^\sigma$.
\end{coro}
\begin{proof}
(1) follows from Lemmas~\ref{lema:lr} and~\ref{lema:mismoscaras}.
(2) follows from Lemma~\ref{lema:mismoscaras} and
equation~\eqref{ecua:sesgaliri} applied to $\rho$ and $\sigma$.
(3) follows from Lemmas~\ref{lema:mismoscaras} and~\ref{lema:skew-standar}.1.
\end{proof}

\begin{nota}
{\em
Let $D=\clase\sigma$, $\pi$ be a composition of $|D|$
and $\rho(1), \dots, \rho(r)$ be partitions.
Then we set $\cara D= \cara{\sigma}$,
${\sf lr}(D,D;\pi)= {\sf lr}(\sigma,\sigma;\pi) $,
$\clirir D = \clirir \sigma$ and $f^D = f^\sigma$.
By the previous results these definitions do not depend on the
representative of $D$.

We denote by $\eseefe Dd$ the set of all diagram classes of size $d$.
}
\end{nota}

The next formula is one of the ingredients in our enhancement of the RT method.

\begin{prop} \label{prop:multiliridiag}
Let $\la$ be a partition of $n$, $\pi$ be a composition of $n$
and $d= |\ol\pi|$.
Then
\begin{equation} \label{ecua:multiliri-diag}
\multiliric \pi = \sum_{D \in \eseefe Dd} {\sf lr}(D,D;\ol\pi)\, \diagra D .
\end{equation}
\end{prop}
\begin{proof}
This follows from Lemma~\ref{lema:itera} and Corollary~\ref{coro:indep}(1)
by grouping together, for each $D\in\eseefe Dd$, all partitions $\alpha$ of
$n-d$ such that $\clase \sesgala = D$ and counting them together by the factor
$\diagra D$.
\end{proof}

\begin{lema} \label{lema:transp}
Let $D$ be a diagram class and $\pi$ be a composition of $|D|$.
Then
\begin{equation*}
{\sf lr}(D,D;\pi) = {\sf lr}(D^\prime,D^\prime;\pi).
\end{equation*}
\end{lema}
\begin{proof}
Recall that $D^\prime$ is the conjugate of $D$ (see Paragraph~\ref{para:clase-diag}).
Let $\sesgala$ be a representative of $D$.
We have, by Definition~\ref{defi:lr}, that
\begin{equation*}
{\sf lr}(D,D;\pi) = \sum_{\rho(1) \vdash \pi_1, \dots,\, \rho(r)\vdash \pi_r}
\left[ \clirir \sesgala \right]^2 .
\end{equation*}
So, equation~\eqref{ecua:conjuliri} implies our claim.
\end{proof}

\begin{teor} \label{teor:poliliritirabo}
Let $\pi$ be a composition of $n$.
Then, there exists a polynomial ${q}_{\ol\pi}(x_C)$ with rational coefficients
in the variables $x_C$, where $C$ runs over the set of connected diagram
classes of size $1 \le |C| \le |\ol\pi|$, such that, for all partitions
$\la$ of $n$, the number $\multiliric \pi$ is obtained from $q_{\ol\pi}(x_C)$
by evaluating each $x_C$ at $\diagra C$.
That is,
\begin{equation*}
\multiliric \pi = q_{\ol\pi}(\diagra C).
\end{equation*}
\end{teor}
\begin{proof}
Let $d=|\ol\pi|$.
Then, by Proposition~\ref{prop:multiliridiag}, we have
\begin{equation*}
\multiliric \pi = \sum_{D \in \eseefe Dd} {\sf lr}(D,D;\ol\pi)\, \diagra D .
\end{equation*}
Let
\begin{equation*}
q_{\ol\pi}(x_C)=\sum_{D\in\eseefe Dd} {\sf lr}(D,D;\ol\pi)\,p_D(x_C) .
\end{equation*}
The claim follows from Theorem~\ref{teor:elbueno}.
\end{proof}

The fact that $\pi$ is a composition and not merely a partition will play
an important role in Section~\ref{sec:evalk}.

\begin{obse} \label{obse:tiras-borde-liri}
{\em
Let $\pi$ be a composition.
If, in the polynomial $q_{\ol\pi}(x_C)$, we substitute each variable $x_C$
by a new variable $t_{B(C)}$, where $B(C)$ is the principal border strip of $C$
(see Paragraph~\ref{para:clase-diag}), we obtain a new polynomial $\wt q_{\ol\pi}(t_B)$
with rational coefficients in (possibly fewer) variables $t_B$, where $B$ runs
over the set of border strip classes $B$ of size $1 \le |B| \le |\ol\pi|$.
Because of Remark~\ref{obse:tiras-borde} and Theorem~\ref{teor:poliliritirabo},
this polynomial satisfies the identity
\begin{equation*}
\multiliric \pi = \wt q_{\ol\pi}(\diagra B).
\end{equation*}
}
\end{obse}

The following formula will be useful in the next section.

\begin{lema} \label{lema:lr-delta}
Let $\pi$ be a composition of $d$.
Then
\begin{equation*}
{\sf lr}(\Delta_d, \Delta_d; \pi) = \binom{d}{\pi} d!
\end{equation*}
\end{lema}
\begin{proof}
Recall that $\Delta_d= \raisebox{.1ex}{\tablau} \sqcup \cdots \sqcup \raisebox{.1ex}{\tablau}\,$
is the disjoint union of $d$ squares.
They can be divided into groups of sizes
$\pi_1, \dots, \pi_r$ in $\binom{d}{\pi}$ ways.
Let $\rho(i)$ be a partition of $\pi_i$.
Since the number of LR tableaux of shape $\Delta_{\pi_i}$
and content $\rho(i)$ is equal to $f^{\rho(i)}$, we have
\begin{equation*}
\clirir {\Delta_d} = \binom{d}{\pi} \prod_{i \in \nume r} f^{\rho(i)}.
\end{equation*}
Then, by Definition~\ref{defi:lr}, we have
\begin{equation*}
{\sf lr}(\Delta_d, \Delta_d; \pi) =
\binom{d}{\pi}^{\! 2}
\prod_{i\in \nume r} \sum_{\rho(i)\vdash \pi_i}\left( f^{\rho(i)} \right)^2 .
\end{equation*}
But for any $n$, one has $\sum_{\la\vdash n} \left( f^\la \right) ^2 = n!$
So, the claim follows.
\end{proof}

The following result has a simple direct proof and is probably well-known.
However, it pops up in our context in a nice way.
It is a sort of generalization of the unimodality property for binomial
coefficients.

\begin{coro}
Let $\la$, $\mu$ be partitions of $d$.
If $\mu \domina \nu$, then $\binom{d}{\mu} \le \binom{d}{\nu}$.
\end{coro}
\begin{proof}
It follows from Theorem~\ref{teor:coefi-dec} and Lemma~\ref{lema:lr-delta}.
\end{proof}

Next we give formulas for ${\sf lr}(D,D; (1^{|D|}))$, ${\sf lr}(D,D; (|D|))$ and bounds
for ${\sf lr}(D,D; \pi)$.

\begin{lema} \label{lema:estandar}
Let $d$, $n$ be integers such that $n>d\ge 1$ and
let $\la$ be a partition of $n$.
Then
\begin{equation*}
\multiliric {(n-d,1^d)} = \sum_{D \in \eseefe Dd} \bigl( f^D \bigr)^2\, p_D( \diagra C).
\end{equation*}
\end{lema}
\begin{proof}
By Proposition~\ref{prop:multiliridiag} we have
\begin{equation*}
\multiliric {(n-d,1^d)} = \sum_{D \in \eseefe Dd} {\sf lr}(D,D; (1^d))\, \diagra D.
\end{equation*}
A LR multitableau of type $(1^d)$ is the same as a standard Young tableau.
Therefore ${\sf lr}(D,D; (1^d)) = \bigl( f^D \bigr)^2$.
The claim follows from Theorem~\ref{teor:elbueno}.
\end{proof}

\begin{lema} \label{lema:retratos}
Let $d$, $n$ be integers such that $n/2  \ge d \ge 1$ and
let $\la$ be a partition of $n$.
Then
\begin{equation*}
\multiliric {(n-d,d)} = \sum_{D \in \eseefe Dd}
\raisebox{-.8ex}{\bigg[} \sum_{\alpha\vdash d} \left( c^D_\alpha \right)^2 \raisebox{-.8ex}{\bigg]}
p_D( \diagra C).
\end{equation*}
\end{lema}
\begin{proof}
By Proposition~\ref{prop:multiliridiag} we have that
\begin{equation*}
\multiliric {(n-d,d)} =
\sum_{D \in \eseefe Dd} {\sf lr}(D,D; (d))\, \diagra D.
\end{equation*}
The claim follows from Definition~\ref{defi:lr} and Theorem~\ref{teor:elbueno}.
\end{proof}

\begin{coro}
Let $\pi$ be a composition of $d$.
Then, for any $D\in \eseefe Dd$ we have
\begin{equation*}
\sum_{\alpha\vdash d} \left( c^D_\alpha \right)^2 \le {\sf lr}(D,D; \pi) \le \bigl( f^D \bigr)^2.
\end{equation*}
\end{coro}
\begin{proof}
It follows from the proofs of the Lemmas~\ref{lema:estandar} and~\ref{lema:retratos}
and Theorem~\ref{teor:coefi-dec}.
\end{proof}

We next illustrate Theorem~\ref{teor:poliliritirabo} by giving the
evaluations of $\multiliric \pi$ for all partitions $\pi$
satisfying $0\le \prof\pi \le 4$.
For the sake of clarity we organize the summands of each
evaluation according the following rules:

(1) We write, as in Proposition~\ref{prop:multiliridiag}, a summand
for each diagram class $D$ of size $\prof\pi$.

(2) If $D$ is not connected, we write $\diagra D$ as $p_D(\diagra C)$.
The formulas we need can be obtained from Lemma~\ref{lema:binom} and
Examples~\ref{ejem:reduc}.

(3) Since, by Lemma~\ref{lema:transp}, both $\diagra D$ and $\diagra{D^\prime}$
have the same coefficient in equation~\eqref{ecua:multiliri-diag}, we group them together.

For $\prof\pi =3$ we order the summands of $\multiliric \pi$ according to the
following list of the 7 diagram classes of size 3.
The first 4 are connected, the remaining 3 are non-connected:
\begin{equation*}
\tablath\, ,\
\raisebox{1.30ex}{\tablatv} \, ;\
\raisebox{.65ex}{\tablatdu}\, ;\
\raisebox{.65ex}{\tablatud}\, ; \
\tabladh \sqcup \tablau \, ,\
\raisebox{.65ex}{\tabladv} \sqcup \tablau\, ;\
\tablatu\, .
\end{equation*}

For $\prof\pi=4$ we order the summands of $\multiliric \pi$ according to the
following list of the 19 diagram classes of size 4.
The first 9 are connected, the remaining 10 are non-connected:
\begin{equation*}
\tablacvc\, ,\ \tablacvuuuu\, ;\ \tablacvtu\, ,\ \tablacvduu\, ;\
\tablacvdd\, ;\
\tablacsut\, ,\ \tablacsuud\, ;\ \tablacsedd\, ,\
\tablacseudu\, .
\end{equation*}
\begin{equation*}
\tablath\sqcup\tablau\, ,\
{\raisebox{1.30ex}{\tablatv}}\sqcup\tablau\, ;\
{\raisebox{.65ex}{\tablatdu}}\sqcup\tablau\, ;\
{\raisebox{.65ex}{\tablatud}}\sqcup\tablau\, ;\
\tablacdhdh\, ,\
{\raisebox{.65ex}{\tabladv}} \sqcup {\raisebox{.65ex}{\tabladv}}\, ;\
\tablacdhdv\, ;
\end{equation*}
\begin{equation*}
\tabladh \sqcup \tablau \sqcup \tablau\, ,\
{\raisebox{.65ex}{\tabladv}} \sqcup \tablau \sqcup \tablau\, ;\
\tablacu\, .
\end{equation*}

Observe that, by Theorem~\ref{teor:coefi-dec},
for the different $\pi$'s of the same depth, the numbers ${\sf lr}(D,D; {\ol\pi})$
grows as $\ol\pi$ decreases in the dominance order.

\begin{evaliri} \label{para:evaliri}
{\em
The evaluations $q_{\ol\nu}(\diagra C)$ for all partitions $\ol\nu$ of size at
most 4 are:
\begin{enumerate}

\item
$\multiliric {(n)} = 1$.

\item
$\multiliric {(n-1,1)} = \diagra{\,\tablau\,}$.

\item
$\multiliric {(n-2,2)} = \diagra{\,\tabladh\,} +
\diagra{\, {\raisebox{.65ex}{\tabladv}} \,} +
2\dbinom{\diagra{\,\tablau\,}}{2}$.

\item
$\multiliric {(n-2,1^2)} = \diagra{\,\tabladh\,} +
\diagra{\, {\raisebox{.65ex}{\tabladv}} \,} +
4\dbinom{\diagra{\,\tablau\,}}{2}$.

\item
$\multiliric {(n-3,3)} = \diagra{\,\tablath\,} +
\diagra{\, {\raisebox{1.30ex}{\tablatv}} \,} +
\diagra{\, {\raisebox{.65ex}{\tablatdu}} \,} +
\diagra{\, {\raisebox{.65ex}{\tablatud}} \,}$\\[.1cm]
\indent
$ \qquad \qquad \qquad
+\, 2\left[ \diagra{\,\tablau\,} -1 \right]
\left[ \diagra{\,\tabladh\,} +
\diagra{\, {\raisebox{.65ex}{\tabladv}} \,} \right] +
6 \dbinom{\diagra{\,\tablau\,}}{3} $.

\item
$\multiliric {(n-3,2,1)} =  \diagra{\,\tablath\,} +
\diagra{\, {\raisebox{1.30ex}{\tablatv}} \,} +
2 \left[
\diagra{\, {\raisebox{.65ex}{\tablatdu}} \,} +
\diagra{\, {\raisebox{.65ex}{\tablatud}} \,}
\right]$\\[.1cm]
\indent
$ \qquad \qquad \qquad
+\, 5\left[ \diagra{\,\tablau\,} -1 \right]
\left[ \diagra{\,\tabladh\,} +
\diagra{\, {\raisebox{.65ex}{\tabladv}} \,} \right] +
18 \dbinom{\diagra{\,\tablau\,}}{3} $.

\item
$\multiliric {(n-3,1^3)} =   \diagra{\,\tablath\,} +
\diagra{\, {\raisebox{1.30ex}{\tablatv}} \,} +
4 \left[
\diagra{\, {\raisebox{.65ex}{\tablatdu}} \,} +
\diagra{\, {\raisebox{.65ex}{\tablatud}} \,}
\right]$\\[.1cm]
\indent
$ \qquad \qquad \qquad
+\, 9\left[ \diagra{\,\tablau\,} -1 \right]
\left[ \diagra{\,\tabladh\,} +
\diagra{\, {\raisebox{.65ex}{\tabladv}} \,} \right] +
36 \dbinom{\diagra{\,\tablau\,}}{3} $.

\item
$\multiliric {(n-4,4)} = \phantom{5} $\\
\indent
$ \qquad \qquad \qquad
\diagra{\, \tablacvc \,} + \diagra{\, \tablacvuuuu\,}
+ \diagra{\, \tablacvtu\,}
+ \diagra{\, \tablacvduu\,}
+ \diagra{\, \tablacvdd\,}  $\\[.1cm]
\indent
$ \qquad \qquad \qquad
+\,  \diagra{\, \tablacsut\,}
+ \diagra{\, \tablacsuud\,}
+ 2 \left[ \diagra{\, \tablacsedd\,}
+ \diagra{\, \tablacseudu\,}\right] $
\\[.1cm]
\indent
$ \qquad \qquad \qquad
+\, 2 \left[ \diagra{\,\tablau\,} -1 \right]
\left[ \diagra{\,\tablath\,} +
\diagra{\, \raisebox{1.3ex}{\tablatv}\,} \right]
$\\[.1cm]
\indent
$ \qquad \qquad \qquad
+\, 3 \left[ \diagra{\,\tablau\,} -2  \right]
\diagra{\, \raisebox{.65ex}{\tablatdu}\,}
+ 3 \left[ \diagra{\,\tablau\,} -1  \right]
\diagra{\, \raisebox{.65ex}{\tablatud}\,}
$\\[.1cm]
\indent
$ \qquad \qquad \qquad
+\, 3 \left[ \dbinom{\diagra{\,\tabladh\,}}{2}
+ \dbinom{\diagra{\, \raisebox{.65ex}{\tabladv}\,}}{2} \right]
+ 2 \left[ \diagra{\, \tabladh\,} \,
\diagra{\, \raisebox{.65ex}{\tabladv} \,}
- \diagra{\, \raisebox{.65ex}{\tablatud} \,} \right]
$\\[.1cm]
\indent
$ \qquad \qquad \qquad
+\, 7  \dbinom{\diagra{\,\tablau\,} - 1}{2}
\left[ \diagra{\,\tabladh\,} +
\diagra{\, {\raisebox{.65ex}{\tabladv}} \,} \right]
+ 24 \dbinom{\diagra{\,\tablau\,}}{4}$.

\item
$\multiliric {(n-4,3,1)} = \phantom{5} $\\
\indent
$ \qquad \qquad \qquad
\diagra{\, \tablacvc \,} + \diagra{\, \tablacvuuuu\,}
+ 2\left[ \diagra{\, \tablacvtu\,}
+ \diagra{\, \tablacvduu\,} \right]
+ \diagra{\, \tablacvdd\,}  $\\[.1cm]
\indent
$ \qquad \qquad \qquad
+\, 2 \left[  \diagra{\, \tablacsut\,}
+ \diagra{\, \tablacsuud\,} \right]
+ 5 \left[ \diagra{\, \tablacsedd\,}
+ \diagra{\, \tablacseudu\,} \right] $
\\[.1cm]
\indent
$ \qquad \qquad \qquad
+\, 5 \left[ \diagra{\,\tablau\,} -1 \right]
\left[ \diagra{\,\tablath\,} +
\diagra{\, \raisebox{1.3ex}{\tablatv}\,} \right]
$\\[.1cm]
\indent
$ \qquad \qquad \qquad
+\, 11 \left[ \diagra{\,\tablau\,} -2  \right]
\diagra{\, \raisebox{.65ex}{\tablatdu}\,}
+ 11 \left[ \diagra{\,\tablau\,} -1  \right]
\diagra{\, \raisebox{.65ex}{\tablatud}\,}
$\\[.1cm]
\indent
$ \qquad \qquad \qquad
+\, 8 \left[ \dbinom{\diagra{\,\tabladh\,}}{2}
+ \dbinom{\diagra{\, \raisebox{.65ex}{\tabladv}\,}}{2} \right]
+ 6 \left[ \diagra{\, \tabladh\,} \,
\diagra{\, \raisebox{.65ex}{\tabladv} \,}
- \diagra{\, \raisebox{.65ex}{\tablatud} \,} \right]
$\\[.1cm]
\indent
$ \qquad \qquad \qquad
+\, 26 \dbinom{\diagra{\,\tablau\,} - 1}{2}
\left[ \diagra{\,\tabladh\,} +
\diagra{\, {\raisebox{.65ex}{\tabladv}} \,} \right]
+ 96 \dbinom{\diagra{\,\tablau\,}}{4}$.

\item
$\multiliric {(n-4,2,2)} =  \phantom{5} $\\
\indent
$ \qquad \qquad \qquad
\diagra{\, \tablacvc \,} + \diagra{\, \tablacvuuuu\,}
+ 3\left[ \diagra{\, \tablacvtu\,}
+ \diagra{\, \tablacvduu\,} \right]
+ 2\diagra{\, \tablacvdd\,}  $\\[.1cm]
\indent
$ \qquad \qquad \qquad
+\, 3 \left[  \diagra{\, \tablacsut\,}
+ \diagra{\, \tablacsuud\,} \right]
+ 7 \left[ \diagra{\, \tablacsedd\,}
+ \diagra{\, \tablacseudu\,} \right] $
\\[.1cm]
\indent
$ \qquad \qquad \qquad
+\, 6 \left[ \diagra{\,\tablau\,} -1 \right]
\left[ \diagra{\,\tablath\,} +
\diagra{\, \raisebox{1.3ex}{\tablatv}\,} \right]
$\\[.1cm]
\indent
$ \qquad \qquad \qquad
+\, 16 \left[ \diagra{\,\tablau\,} -2  \right]
\diagra{\, \raisebox{.65ex}{\tablatdu}\,}
+ 16 \left[ \diagra{\,\tablau\,} -1  \right]
\diagra{\, \raisebox{.65ex}{\tablatud}\,}
$\\[.1cm]
\indent
$ \qquad \qquad \qquad
+\, 12 \left[ \dbinom{\diagra{\,\tabladh\,}}{2}
+ \dbinom{\diagra{\, \raisebox{.65ex}{\tabladv}\,}}{2} \right]
+ 10 \left[ \diagra{\, \tabladh\,} \,
\diagra{\, \raisebox{.65ex}{\tabladv} \,}
- \diagra{\, \raisebox{.65ex}{\tablatud} \,} \right]
$\\[.1cm]
\indent
$ \qquad \qquad \qquad
+\, 38 \dbinom{\diagra{\,\tablau\,} - 1}{2}
\left[ \diagra{\,\tabladh\,} +
\diagra{\, {\raisebox{.65ex}{\tabladv}} \,} \right]
+ 144 \dbinom{\diagra{\,\tablau\,}}{4}$.

\item
$\multiliric {(n-4,2,1^2)} = \phantom{5} $\\
\indent
$ \qquad \qquad \qquad
\diagra{\, \tablacvc \,} + \diagra{\, \tablacvuuuu\,}
+ 5\left[ \diagra{\, \tablacvtu\,}
+ \diagra{\, \tablacvduu\,} \right]
+ 2\diagra{\, \tablacvdd\,}  $\\[.1cm]
\indent
$ \qquad \qquad \qquad
+\, 5 \left[  \diagra{\, \tablacsut\,}
+ \diagra{\, \tablacsuud\,} \right]
+ 13 \left[ \diagra{\, \tablacsedd\,}
+ \diagra{\, \tablacseudu\,} \right] $
\\[.1cm]
\indent
$ \qquad \qquad \qquad
+\, 10 \left[ \diagra{\,\tablau\,} -1 \right]
\left[ \diagra{\,\tablath\,} +
\diagra{\, \raisebox{1.3ex}{\tablatv}\,} \right]
$\\[.1cm]
\indent
$ \qquad \qquad \qquad
+\, 32 \left[ \diagra{\,\tablau\,} -2  \right]
\diagra{\, \raisebox{.65ex}{\tablatdu}\,}
+ 32 \left[ \diagra{\,\tablau\,} -1  \right]
\diagra{\, \raisebox{.65ex}{\tablatud}\,}
$\\[.1cm]
\indent
$ \qquad \qquad \qquad
+\, 20 \left[ \dbinom{\diagra{\,\tabladh\,}}{2}
+ \dbinom{\diagra{\, \raisebox{.65ex}{\tabladv}\,}}{2} \right]
+ 18 \left[ \diagra{\, \tabladh\,} \,
\diagra{\, \raisebox{.65ex}{\tabladv} \,}
- \diagra{\, \raisebox{.65ex}{\tablatud} \,} \right]
$\\[.1cm]
\indent
$ \qquad \qquad \qquad
+\, 74 \dbinom{\diagra{\,\tablau\,} - 1}{2}
\left[ \diagra{\,\tabladh\,} +
\diagra{\, {\raisebox{.65ex}{\tabladv}} \,} \right]
+ 288 \dbinom{\diagra{\,\tablau\,}}{4}$.

\item
$\multiliric {(n-4,1^4)} =
 \phantom{5} $\\
\indent
$ \qquad \qquad \qquad
\diagra{\, \tablacvc \,} + \diagra{\, \tablacvuuuu\,}
+ 9\left[ \diagra{\, \tablacvtu\,}
+ \diagra{\, \tablacvduu\,} \right]
+ 4\diagra{\, \tablacvdd\,}  $\\[.1cm]
\indent
$ \qquad \qquad \qquad
+\, 9 \left[  \diagra{\, \tablacsut\,}
+ \diagra{\, \tablacsuud\,} \right]
+ 25 \left[ \diagra{\, \tablacsedd\,}
+ \diagra{\, \tablacseudu\,} \right] $
\\[.1cm]
\indent
$ \qquad \qquad \qquad
+\, 16 \left[ \diagra{\,\tablau\,} -1 \right]
\left[ \diagra{\,\tablath\,} +
\diagra{\, \raisebox{1.3ex}{\tablatv}\,} \right]
$\\[.1cm]
\indent
$ \qquad \qquad \qquad
+\, 64 \left[ \diagra{\,\tablau\,} -2  \right]
\diagra{\, \raisebox{.65ex}{\tablatdu}\,}
+ 64 \left[ \diagra{\,\tablau\,} -1  \right]
\diagra{\, \raisebox{.65ex}{\tablatud}\,}
$\\[.1cm]
\indent
$ \qquad \qquad \qquad
+\, 36 \left[ \dbinom{\diagra{\,\tabladh\,}}{2}
+ \dbinom{\diagra{\, \raisebox{.65ex}{\tabladv}\,}}{2} \right]
+ 36 \left[ \diagra{\, \tabladh\,} \,
\diagra{\, \raisebox{.65ex}{\tabladv} \,}
- \diagra{\, \raisebox{.65ex}{\tablatud} \,} \right]
$\\[.1cm]
\indent
$ \qquad \qquad \qquad
+\, 144 \dbinom{\diagra{\,\tablau\,} - 1}{2}
\left[ \diagra{\,\tabladh\,} +
\diagra{\, {\raisebox{.65ex}{\tabladv}} \,} \right]
+ 576 \dbinom{\diagra{\,\tablau\,}}{4}$.
\end{enumerate}
}
\end{evaliri}

\begin{obses}
{\em
(1) We think of the binomial coefficient $\binom{t}{n}$ as a polynomial in $t$:
\begin{equation*}
\binom{t}{n} = \frac{t(t-1) \cdots (t-n+1)}{n!}.
\end{equation*}

(2) The twelve polynomials above, with the sole exception of the polynomial
from example~8, have integer coefficients.
}
\end{obses}

\section{The Kronecker coefficients {\mathversion{bold}$\ccuad \nu$}}
\label{sec:evalk}

This section is the core of the paper.
Theorem~\ref{teor:kron-combin} is an enhancement of the RT method that gives a
closed combinatorial formula (up to signs) of Kronecker coefficients.
We will show its utility in Sections~\ref{sec:saxl} and~\ref{sec:estab}.
Theorem~\ref{teor:polikrontirabo} shows a new phenomenon of Kronecker coefficients,
namely, that each coefficient of the form $\ccuad{(n-d,\ol\nu)}$
can be computed by the evaluation of a polynomial $k_{\ol\nu}$ in as many
variables as connected diagram classes $C$ of size $1 \le |C| \le |\ol\nu|$.
Theorem~\ref{teor:tiras-borde-kron} shows that the polynomial $k_{\ol\nu}$
can be changed to a polynomial $\wt k_{\ol\nu}$ in variables $t_B$, one for each
border strip class $B$ of size $1 \le |B| \le |\ol\nu|$.
Also an evaluation of $\wt k_{\ol\nu}$ yields $\ccuad{(n-d,\ol\nu)}$
for all partitions $\la$ of some integer $n \ge \nu_2 + |\ol\nu|$.
Several results about these polynomials are presented.
In Paragraph~\ref{para:evakron} we compute the polynomials $k_{\ol\nu}$ for all
$|\ol\nu| \le 4$.

\begin{teor} \label{teor:kron-combin}
Let $\ol\nu = (\nu_2, \dots, \nu_r)$ be a partition of $d$ and
$n\ge d + \nu_2$.
Then, for any partition $\la$ of $n$, we have
\begin{equation} \label{ecua:kron-merabuena}
{\sf g}(\la,\la, (n-d,\ol\nu)) = \sum_{k=0}^d \sum_{D \in \eseefe Dk}
\sum_{ \substack{T\in \tirabo {\wt\nu} \\ e(T) = |D|}} \signo T\, {\sf lr}(D,D; \ol\tau(T))\, \diagra D.
\end{equation}
\end{teor}
\begin{proof}
By Proposition~\ref{prop:robtau} we have
\begin{equation*}
{\sf g}(\la,\la, (n-d,\ol\nu)) =
\sum_{T\in \tirabo{\wt\nu}} {\sf sign}(T)\, \multiliric {\tau(B_n(T))} .
\end{equation*}
Recall that for each $T \in \tirabo{\wt\nu}$, one has $e(T)=|\ol\tau(T)| \le d$
(Definition~\ref{defi:nuevo-cont}).
Since, by Lemma~\ref{lema:biyec-tiraborde}, $\ol\tau(B_n(T)) = \ol\tau(T)$,
it follows from Proposition~\ref{prop:multiliridiag} that
\begin{equation*}
\multiliric{\tau(B_n(T))} = \sum_{D \in \eseefe D{e(T)}} {\sf lr}(D,D;\ol\tau(T))\, \diagra D .
\end{equation*}
Since $e(T) \le d$, the theorem follows from the two previous identities.
\end{proof}

\begin{teor} \label{teor:polikrontirabo}
Let $\ol\nu$ be a partition of $d$.
Then there exists a polynomial with rational coefficients $k_{\ol\nu}(x_C)$
in the variables $x_C$, where $C$ runs over the set of connected
diagram classes of size $1 \le |C| \le d$, such that, for all $n \ge d + \nu_2$
and all partitions $\la$ of $n$, the Kronecker coefficient $\ccuad {(n-d, \ol\nu)}$
is obtained from $k_{\ol\nu}(x_C)$ by evaluating each $x_C$ at $\diagra C$,
that is,
\begin{equation*}
\ccuad {(n-d, \ol\nu)} = k_{\ol\nu}(\diagra C).
\end{equation*}
\end{teor}
\begin{proof}
Let
\begin{equation*}
k_{\ol\nu}(x_C) =
\sum_{ D, \, |D| \le d}\, \sum_{ \substack{T\in \tirabo {\wt\nu} \\ e(T) = |D|}}
\signo T\, {\sf lr}(D,D; \ol\tau(T))\, p_D(x_C).
\end{equation*}
Then, the proof follows from Theorems~\ref{teor:elbueno} and~\ref{teor:kron-combin}.
\end{proof}

Alternatively, the polynomial $k_{\ol\nu}$ can be defined from the polynomials $q_{\ol\pi}$
defined in the proof of Theorem~\ref{teor:poliliritirabo}.

\begin{lema} \label{lema:todo-pol}
Let $\ol\nu$ be a partition.
Then
\begin{equation*}
k_{\ol\nu}(x_C) = \sum_{T\in \tirabo{\wt\nu}} {\sf sign}(T)\, q_{\ol\tau(T)}(x_C).
\end{equation*}
\end{lema}
\begin{proof}
This follows from
Theorems~\ref{teor:poliliritirabo} and~\ref{teor:polikrontirabo}.
\end{proof}

Theorem~\ref{teor:polikrontirabo} can be restated in the following way:

\begin{teor} \label{teor:tiras-borde-kron}
Let $\ol\nu$ be a partition of $d$.
Then there exists a polynomial with rational coefficients $\wt k_{\ol\nu}(t_B)$
in the variables $t_B$, where $B$ runs over the set of border strip
classes of size $1 \le |B| \le d$, such that, for all $n \ge d + \nu_2$
and all partitions $\la$ of $n$, the Kronecker coefficient $\ccuad {(n-d, \ol\nu)}$
is obtained from $\wt k_{\ol\nu}(t_B)$ by evaluating each $t_B$ at $\diagra B$,
that is,
\begin{equation*}
\ccuad {(n-d, \ol\nu)} = \wt k_{\ol\nu}(\diagra B).
\end{equation*}
\end{teor}
\begin{proof}
In the polynomial $k_{\ol\nu}(x_C)$, we substitute each variable $x_C$ by a
new variable $t_{B(C)}$, where $B(C)$ is the principal border strip of $C$.
Then we obtain a new polynomial $\wt k_{\ol\nu}(t_B)$ with rational coefficients
in (possibly fewer) variables $t_B$, where $B$ runs over the set of
border strip classes $B$ of size $1 \le |B| \le d$.
The claim follows from Lemma~\ref{lema:todo-pol},
Remark~\ref{obse:tiras-borde-liri} and Proposition~\ref{prop:robtau}.
\end{proof}

The following result appears in~\cite[p.~23]{vsqkron}.
It was rediscovered in~\cite{pp2}.
Corollary 2.1 in~\cite{mani} is a particular case of the next lemma and
can be derived from it.

\begin{lema} \label{lema:kron-dos-num-parti}
Let $n$, $d$ be such that $n\ge 2d$.
If $\la = (a^b)$ is a partition of $n$, then
\begin{equation*}
\ccuad {(n-d, d)} = \#\{ \alpha \vdash d \mid\alpha \subseteq \la\}
- \#\{ \beta \vdash d-1 \mid \beta \subseteq \la \}.
\end{equation*}
\end{lema}
\begin{proof}
For each partition $\alpha = \vector \alpha t$, denote
$\rotap \alpha = (\alpha_1^t)/(\alpha_1 - \alpha_t, \dots, \alpha_1 - \alpha_2)$.
Then $\rotap \alpha$ is a skew diagram that is obtained by rotating 180 degrees the diagram
of $\alpha$.
By Theorem~\ref{teor:kron-combin}
\begin{equation*}
{\sf g}(\la,\la, (n-d, d)) =
\sum_{D,\ |D| = d} {\sf lr}(D,D; (d))\, \diagra D -
\sum_{E,\ |E| = d-1} {\sf lr}(E,E; (d-1))\, \diagra E.
\end{equation*}
Since $\la$ is a rectangle, $\diagra D \neq 0$ if and only if there is a partition
$\alpha \subseteq \la$ such that $D = \nume{\rotap \alpha}$.
If any of these two conditions hold, we have $\diagra D = 1$.
By the Littlewood-Richardson rule $\cara\alpha = \cara{\rotap \alpha}$.
Therefore, if $D$ is any diagram of size $k$ with $\diagra D \neq 0$, we have
\begin{equation*}
{\sf lr}(D,D; (k))\,  \diagra D = \langle \cara\alpha \otimes \cara\alpha , \permu{(k)} \rangle = 1.
\end{equation*}
From this the proposition follows.
\end{proof}

In some cases the coefficient of $\diagra D$ in the expansion of
${\sf g}(\la,\la, (n-d,\ol\nu))$ in Theorem~\ref{teor:kron-combin}
can be computed by a simpler formula.

\begin{prop} \label{prop:coefi-rld}
Let $\ol\nu = (\nu_2, \dots, \nu_r)$ be a partition of $d$ and
$n\ge d + \nu_2$.
Then, for any partition $\la$ of $n$ and any diagram class $D$ of size $d$,
the coefficient of $\diagra D$ in equation~\eqref{ecua:kron-merabuena} is
\begin{equation*}
\sum_{\alpha, \beta \vdash d} c^D_\alpha c^D_\beta {\sf g}(\alpha, \beta, \ol\nu).
\end{equation*}
\end{prop}
\begin{proof}
By Theorem~\ref{teor:kron-combin} the coefficient of $\diagra D$ in
equation~\eqref{ecua:kron-merabuena} is
\begin{equation*}
\sum_{ \substack{T\in \tirabo {\wt\nu} \\ e(T) = d}} \signo T\, {\sf lr}(D,D; \ol\tau(T)).
\end{equation*}
But for each $T\in \tirabo {\wt\nu}$ with $e(T)=d$ the special border strip that contains
the square $(1,\nu_2)$ must be contained in the first row of $\wt\nu$.
Let $\ol T$ be the tableau obtained from $T$ by deleting the first row.
Then, $\ol T \in \tirabo {\ol\nu}$, $\signo{\ol T} = \signo T$ and
$\gamma(\ol T) = \ol\tau(T)$.
Therefore, the coefficient of $\diagra D$ in equation~\eqref{ecua:kron-merabuena} is,
by Theorem~\ref{teor:er},
\begin{equation} \label{ecua:aux}
\sum_{R \in \tirabo{\ol\nu}} \signo R\, {\sf lr}(D,D; \gamma(R))
= \sum_{\mu \vdash d} \kos\mu{\ol\nu}^{(-1)}\, {\sf lr}(D,D; \mu).
\end{equation}
And, because of Lemma~\ref{lema:lr} and equation~\eqref{ec:jt},
the term on the right of equation~\eqref{ecua:aux} is
$\left< \cara D \otimes \cara D , \cara{\ol\nu} \right>$.
The claim follows after expressing each $\cara D$ as a linear combination
of irreducible characters and applying Lemma~\ref{lema:skew-standar}.2.
\end{proof}

\begin{coro} \label{coro:coefi-rld}
Let $\ol\nu = (\nu_2, \dots, \nu_r)$ be a partition of $d$ and
$n\ge d + \nu_2$.
Then, for any partitions $\la$ of $n$ and $\mu$ of $d$,
the coefficient of $\diagra {\clase \mu}$ in equation~\eqref{ecua:kron-merabuena} is
${\sf g}(\mu, \mu, \ol\nu)$.
\end{coro}
\begin{proof}
If $D = \clase\mu$ and $\alpha \vdash d$,
the Littlewood-Richardson coefficient $c^D_\alpha$ is different from $0$
if and only if $\alpha = \mu$.
Since $c^D_\mu = 1$, the claim follows.
\end{proof}

Equation~\eqref{ecua:kron-merabuena} takes a simpler form in the case $\ol\nu = (1^d)$.

\begin{prop} \label{prop:kron-escuadra}
Let $n$, $d\in \natural$ be such that $n >d$.
Then for any partition $\la$ of $n$
\begin{equation} \label{ecua:kron-escuadra}
{\sf g}(\la, \la, (n-d, 1^d))=
\sum_{k=0}^d (-1)^{d-k} \sum_{D \in \eseefe Dk}
\left[ \sum_{\alpha \vdash k} c^D_\alpha c^D_{\alpha^\prime} \right] \diagra D.
\end{equation}
\end{prop}
\begin{proof}
Let $k\in \nume d$ and $D$ be a diagram class of size $k$.
Let $T\in \tirabo{(1^{d+1})}$ with  $e(T)= k$.
If we remove the uppermost special border strip of $T$, we get a tableau
$R\in \tirabo{(1^k)}$ with $\signo T = (-1)^{d-k} \signo R$.
Thus, we have
\begin{equation} \label{ecua:aux2}
\sum_{ \substack{T\in \tirabo {\wt\nu} \\ e(T) = |D|}} \signo T\, {\sf lr}(D,D; \ol\tau(T)) =
(-1)^{d-k} \sum_{R \in \tirabo{(1^k)}} \signo R \, {\sf lr}(D,D; \gamma(R)).
\end{equation}
By Lemma~\ref{lema:lr} and equation~\eqref{ec:jtdos} we have that the sum
on the right of equation~\eqref{ecua:aux2} is
$ \left< \cara D \otimes \cara D, \cara {(1^k)} \right>$.
By Lemma~\ref{lema:skew-standar}.2 this number is equal to
$\sum_{\alpha \vdash k} c^D_\alpha c^D_{\alpha^\prime}$.
Note that there is a tableau $T\in \tirabo{(1^{d+1})}$ with $e(T) = 0$ and $\signo T = (-1)^d$.
Then, by grouping together all diagram classes of the same size in
equation~\eqref{ecua:kron-merabuena}, we prove the proposition.
\end{proof}

We recover with our techniques the following result which appears in~\cite[\S~6]{pp1}.
Pak and Panova used Lemmas~\ref{lema:kron-dos-num-parti} and~\ref{lema:escuadra-parti}
to prove some results on unimodality.
Corollary~2.2 in~\cite{mani} follows from the next lemma and the well-known
bijection between the set of self conjugate partitions of $k$ and the set
of partitions of $k$ with distinct odd parts.

\begin{lema} \label{lema:escuadra-parti}
Let $d$, $n\in \natural$ be such that $n>d$.
If $\la = (a^b)$ is a rectangle partition of $n$, then
\begin{equation*}
\ccuad {(n-d,1^d)} =
\sum_{k=0}^d (-1)^{d-k} \#\{ \alpha \vdash k \mid \alpha \subseteq \la \text{ and } \alpha = \alpha^\prime \}.
\end{equation*}
\end{lema}
\begin{proof}
Let $D\in \eseefe Dk$ be such that $\diagra D \neq 0$.
Then there is a partition $\beta\subseteq \la$ such that $D = \nume {\rotap \beta}$
and $\diagra D = 1$ (see the proof of Lemma~\ref{lema:kron-dos-num-parti}).
Recall also that $\cara D = \cara{\rotap \beta} = \cara\beta$.
Therefore
\begin{equation*}
\sum_{\alpha \vdash k} c^D_\alpha c^D_{\alpha^\prime} = \left< \cara\beta\otimes\cara\beta, \cara{(1^n)} \right>
=\left< \cara\beta, \cara{\beta^\prime} \right>,
\end{equation*}
which is, by the orthogonality relations, the Kronecker delta $\delta_{\beta, \beta^\prime}$.
Then the summand for $D$ in equation~\eqref{ecua:kron-escuadra} is non-zero if and only if
$\beta\subseteq \la$ and $\beta = \beta^\prime$.
If any of these conditions hold we have
$\left[ \sum_{\alpha\vdash k} c^D_\alpha c^D_{\alpha^\prime} \right] \diagra D = 1$.
From this the lemma follows.
\end{proof}

\begin{lema} \label{lema:coefi-rldeltak}
Let $\ol\nu = (\nu_2, \dots, \nu_r)$ be a partition of $d$ and
$n\ge d + \nu_2$.
Then, for any partition $\la$ of $n$ and any $k \in \nume d$,
the coefficient of $\diagra {\Delta_k}$ in equation~\eqref{ecua:kron-merabuena} is
\begin{equation*}
\Biggl[\sum_{ T\in \tirabo {\wt\nu},\  e(T) = k}
\signo T\, \binom{k}{\ol\tau(T)} \Biggr] k!
\end{equation*}
\end{lema}
\begin{proof}
By Theorem~\ref{teor:kron-combin}, the coefficient of $\diagra {\Delta_k}$ in
equation~\eqref{ecua:kron-merabuena} is
\begin{equation*}
\sum_{ \substack{T\in \tirabo {\wt\nu} \\ e(T) = k}}
\signo T\, {\sf lr}(\Delta_k, \Delta_k; \ol\tau(T)).
\end{equation*}
The claim follows from Lemma~\ref{lema:lr-delta}.
\end{proof}

Some particular instances of the previous lemma are given in the following:

\begin{prop} \label{prop:coefi-rldeltakdos}
Let $\ol\nu = (\nu_2, \dots, \nu_r)$ be a partition of $d$ and
$n\ge d + \nu_2$.
Then, for any partition $\la$ of $n$ and any $k \in \nume d$,
the coefficient of $\diagra {\Delta_k}$ in equation~\eqref{ecua:kron-merabuena} is:

(1) $f^{\ol\nu} d!$, if $k=d$ and $\ol\nu$ is any partition.

(2) $(-1)^{d-k} k! \binom{k}{m-1}$, if $k < d$ and $\ol\nu = (m, 1^{d-m})$ with $m \in \nume d$.
\end{prop}
\begin{proof}
If $k=d$, we have, by a similar argument as the one used in the proof of
Proposition~\ref{prop:coefi-rld}, that
\begin{equation*}
\sum_{T\in \tirabo {\wt\nu}, \ e(T) = d}
\signo T\, \binom{k}{\ol\tau(T)} =
\sum_{\mu \vdash d} K^{(-1)}_{\mu,\ol\nu} \binom{d}{\mu}.
\end{equation*}
Then, (1) follows from Lemma~\ref{lema:coefi-rldeltak} and
Lemma~\ref{lema:kostka-multinom}.

For (2), let $T\in \tirabo {\wt\nu}$ with $e(T)=k$.
Assume $m>1$ (the case $m=1$ is similar, but simpler).
Since $e(T) <d$, the special border strip $\zeta$ of $T$ that contains the square
$(1,m)$ has length greater than $m$ and thus cannot contain $(1,1)$
(otherwise $\zeta$ would contain $(2,1)$ and there would be no space for another
special border strip to cover the squares $(2,2), \dots, (2,m)$).
We conclude that $\zeta$ must go through $(1,m)$, $(2,m)$ and $(2,1)$.
Then, there must be a special border strip of length $m-1$ contained in the first row of $T$.
The remaining special border strips of $T$ are contained in the in the first column,
and the total sum of its sizes is $k-m+1$.
If $k < m-1$ there is no room for $\zeta$, so there is no
$T\in \tirabo {\wt\nu}$ with $e(T)=k$.
Therefore, by Lemma~\ref{lema:coefi-rldeltak}, the desired coefficient is zero.
So, we assume from now on that $k \ge m-1$.
Note that in this case $\zeta$ is contained in exactly $d-k+1$ rows of $T$.
Therefore
\begin{equation*}
\sum_{T\in \tirabo {\wt\nu}, \ e(T) = k}
\signo T\, \binom{k}{\ol\tau(T)} =
(-1)^{d-k} \sum_{R\in \tirabo {(1^{k-m+1})}} \signo R\, \binom{k}{m-1,\ \gamma(R)}.
\end{equation*}
Since $\binom{k}{m-1,\ \gamma(R)} = \binom{k}{m-1} \binom{k-m+1}{\gamma(R)}$,
and since, by Lemma~\ref{lema:kostka-multinom}, we have the identity
\begin{equation*}
 \sum_{R\in \tirabo {(1^{k-m+1})}} \signo R\, \binom{k-m+1}{\gamma(R)} = 1,
\end{equation*}
we get that
\begin{equation*}
\sum_{T\in \tirabo {\wt\nu}, \ e(T) = k} \signo T\, \binom{k}{\ol\tau(T)} =
(-1)^{d-k} \binom{k}{m- 1}.
\end{equation*}
Then, (2) follows from Lemma~\ref{lema:coefi-rldeltak}.
\end{proof}

Since the polynomials $k_{\ol\nu}(x_C)$
depend neither on $\la$ nor on $\nu_1$, we can obtain quite
general formulas for components of Kronecker squares of small depth.
The following evaluations of the polynomials $k_{\ol\nu}(x_C)$ are obtained from
Lemma~\ref{lema:todo-pol} and the formulas in number~\ref{para:evaliri}.
Some of its coefficients can also be computed directly from the results in this chapter.

\begin{evakron} \label{para:evakron}
{\em
The evaluations $k_{\ol\nu}(\diagra C)$ for all partitions $\ol\nu$ of size at
most 4 are:
\begin{enumerate}

\item
$\ccuad{(n)} = 1$.

\item
$\ccuad{(n-1,1)} = \diagra{\,\tablau\,} -1$.

\item
$\ccuad{(n-2,2)} = \diagra{\,\tabladh\,} +
\diagra{\, {\raisebox{.65ex}{\tabladv}} \,} +
2\dbinom{\diagra{\,\tablau\,}}{2} - \diagra{\,\tablau\,}$.

\item
$\ccuad{(n-2,1^2)} = \left[ \diagra{\,\tablau\,} -1 \right]^2=
2\dbinom{\diagra{\,\tablau\,}}{2} - \diagra{\,\tablau\,}+1$.

\item
$\ccuad{(n-3,3)} = \diagra{\,\tablath\,} +
\diagra{\, {\raisebox{1.30ex}{\tablatv}} \,} +
\diagra{\, {\raisebox{.65ex}{\tablatdu}} \,} +
\diagra{\, {\raisebox{.65ex}{\tablatud}} \,}$\\[.1cm]
\indent
$ \qquad \qquad \quad
+\, \left[ 2\diagra{\,\tablau\,} -3 \right]
\left[ \diagra{\,\tabladh\,} +
\diagra{\, {\raisebox{.65ex}{\tabladv}} \,} \right] +
6 \dbinom{\diagra{\,\tablau\,}}{3}
-2 \dbinom{\diagra{\,\tablau\,}}{2}$.

\item
$\ccuad{(n-3,2,1)} =
\diagra{\, {\raisebox{.65ex}{\tablatdu}} \,} +
\diagra{\, {\raisebox{.65ex}{\tablatud}} \,} +
\, \left[3 \diagra{\,\tablau\,} -4 \right]
\left[ \diagra{\,\tabladh\,} +
\diagra{\, {\raisebox{.65ex}{\tabladv}} \,} \right]
$\\[.1cm]
\indent
$ \qquad \qquad \quad
+\,12 \dbinom{\diagra{\,\tablau\,}}{3}
-4\dbinom{\diagra{\,\tablau\,}}{2} + \diagra{\,\tablau\,}$.

\item
$\ccuad{(n-3,1^3)} =
\diagra{\, {\raisebox{.65ex}{\tablatdu}} \,} +
\diagra{\, {\raisebox{.65ex}{\tablatud}} \,}
+ \left[ \diagra{\,\tablau\,} - 1 \right]
\left[ \diagra{\,\tabladh\,} +
\diagra{\, {\raisebox{.65ex}{\tabladv}} \,} \right]
$\\[.1cm]
\indent
$ \qquad \qquad \quad
+\, 6 \dbinom{\diagra{\,\tablau\,}}{3}
- 2\dbinom{\diagra{\,\tablau\,}}{2} + \diagra{\,\tablau\,} -1$.

\item
$\ccuad{(n-4,4)} =  \phantom{5} $\\
\indent
$ \qquad \qquad \qquad
\diagra{\, \tablacvc \,} + \diagra{\, \tablacvuuuu\,}
+ \diagra{\, \tablacvtu\,}
+ \diagra{\, \tablacvduu\,}
+ \diagra{\, \tablacvdd\,}  $\\[.1cm]
\indent
$ \qquad \qquad \qquad
+\,  \diagra{\, \tablacsut\,}
+ \diagra{\, \tablacsuud\,}
+ 2 \left[ \diagra{\, \tablacsedd\,}
+ \diagra{\, \tablacseudu\,} \right] $
\\[.1cm]
\indent
$ \qquad \qquad \qquad
+  \left[ 2\diagra{\,\tablau\,} -3 \right]
\left[ \diagra{\,\tablath\,} +
\diagra{\, \raisebox{1.3ex}{\tablatv}\,} \right]
$\\[.1cm]
\indent
$ \qquad \qquad \qquad
+  \left[ 3\diagra{\,\tablau\,} - 7  \right]
\diagra{\, \raisebox{.65ex}{\tablatdu}\,}
+  \left[ 3\diagra{\,\tablau\,} - 4  \right]
\diagra{\, \raisebox{.65ex}{\tablatud}\,}
$\\[.1cm]
\indent
$ \qquad \qquad \qquad
+\, 3 \left[ \dbinom{\diagra{\,\tabladh\,}}{2}
+ \dbinom{\diagra{\, \raisebox{.65ex}{\tabladv}\,}}{2} \right]
+ 2 \left[ \diagra{\, \tabladh\,} \,
\diagra{\, \raisebox{.65ex}{\tabladv} \,}
- \diagra{\, \raisebox{.65ex}{\tablatud} \,} \right]
$\\[.1cm]
\indent
$ \qquad \qquad \qquad
+  \left[ 7\dbinom{\diagra{\,\tablau\,}}{2} -9 \diagra{\,\tablau\,}
+ 9 \right]
\left[ \diagra{\,\tabladh\,} +
\diagra{\, {\raisebox{.65ex}{\tabladv}} \,} \right]
$\\[.1cm]
\indent
$ \qquad \qquad \qquad
+\, 24 \dbinom{\diagra{\,\tablau\,}}{4} -6\dbinom{\diagra{\,\tablau\,}}{3}$.

\item
$\ccuad{(n-4,3,1)} =  \phantom{5} $\\
\indent
$ \qquad \qquad \qquad
\diagra{\, \tablacvtu\,}
+ \diagra{\, \tablacvduu\,}
+ \diagra{\, \tablacsut\,}
+ \diagra{\, \tablacsuud\,}
$\\[.1cm]
\indent
$ \qquad \qquad \qquad
+\, 3 \left[ \diagra{\, \tablacsedd\,}
+ \diagra{\, \tablacseudu\,} \right] $
\\[.1cm]
\indent
$ \qquad \qquad \qquad
+  \left[ 3\diagra{\,\tablau\,} - 4 \right]
\left[ \diagra{\,\tablath\,} +
\diagra{\, \raisebox{1.3ex}{\tablatv}\,} \right]
$\\[.1cm]
\indent
$ \qquad \qquad \qquad
+  \left[ 8\diagra{\,\tablau\,} - 18  \right]
\diagra{\, \raisebox{.65ex}{\tablatdu}\,}
+  \left[ 8\diagra{\,\tablau\,} - 10  \right]
\diagra{\, \raisebox{.65ex}{\tablatud}\,}
$\\[.1cm]
\indent
$ \qquad \qquad \qquad
+\, 5 \left[ \dbinom{\diagra{\,\tabladh\,}}{2}
+ \dbinom{\diagra{\, \raisebox{.65ex}{\tabladv}\,}}{2} \right]
+ 4 \left[ \diagra{\, \tabladh\,} \,
\diagra{\, \raisebox{.65ex}{\tabladv} \,}
- \diagra{\, \raisebox{.65ex}{\tablatud} \,} \right]
$\\[.1cm]
\indent
$ \qquad \qquad \qquad
+  \left[ 19\dbinom{\diagra{\,\tablau\,}}{2} -24 \diagra{\,\tablau\,}
+ 25 \right]
\left[ \diagra{\,\tabladh\,} +
\diagra{\, {\raisebox{.65ex}{\tabladv}} \,} \right]
$\\[.1cm]
\indent
$ \qquad \qquad \qquad
+\, 72 \dbinom{\diagra{\,\tablau\,}}{4} -18\dbinom{\diagra{\,\tablau\,}}{3}
+2\dbinom{\diagra{\,\tablau\,}}{2}$.

\item
$\ccuad{(n-4,2,2)} =  \phantom{5} $\\
\indent
$ \qquad \qquad \qquad
\diagra{\, \tablacvtu\,}
+ \diagra{\, \tablacvduu\,}
+ \diagra{\, \tablacvdd\,}  $\\[.1cm]
\indent
$ \qquad \qquad \qquad
+\,  \diagra{\, \tablacsut\,}
+ \diagra{\, \tablacsuud\,}
+ 2 \left[ \diagra{\, \tablacsedd\,}
+ \diagra{\, \tablacseudu\,} \right] $
\\[.1cm]
\indent
$ \qquad \qquad \qquad
+  \left[ \diagra{\,\tablau\,} -1 \right]
\left[ \diagra{\,\tablath\,} +
\diagra{\, \raisebox{1.3ex}{\tablatv}\,} \right]
$\\[.1cm]
\indent
$ \qquad \qquad \qquad
+  \left[ 5\diagra{\,\tablau\,} - 11  \right]
\diagra{\, \raisebox{.65ex}{\tablatdu}\,}
+  \left[ 5\diagra{\,\tablau\,} - 6  \right]
\diagra{\, \raisebox{.65ex}{\tablatud}\,}
$\\[.1cm]
\indent
$ \qquad \qquad \qquad
+\, 4 \left[ \dbinom{\diagra{\,\tabladh\,}}{2}
+ \dbinom{\diagra{\, \raisebox{.65ex}{\tabladv}\,}}{2} \right]
+ 4 \left[ \diagra{\, \tabladh\,} \,
\diagra{\, \raisebox{.65ex}{\tabladv} \,}
- \diagra{\, \raisebox{.65ex}{\tablatud} \,} \right]
$\\[.1cm]
\indent
$ \qquad \qquad \qquad
+  \left[ 12\dbinom{\diagra{\,\tablau\,}}{2} -15 \diagra{\,\tablau\,}
+ 15 \right]
\left[ \diagra{\,\tabladh\,} +
\diagra{\, {\raisebox{.65ex}{\tabladv}} \,} \right]
$\\[.1cm]
\indent
$ \qquad \qquad \qquad
+\, 48 \dbinom{\diagra{\,\tablau\,}}{4} -12\dbinom{\diagra{\,\tablau\,}}{3}
+2 \dbinom{\diagra{\,\tablau\,}}{2}$.

\item
$\ccuad{(n-4,2,1^2)} =  \phantom{5} $\\
\indent
$ \qquad \qquad \qquad
\diagra{\, \tablacvtu\,}
+ \diagra{\, \tablacvduu\,}
$\\[.1cm]
\indent
$ \qquad \qquad \qquad
+\,  \diagra{\, \tablacsut\,}
+ \diagra{\, \tablacsuud\,}
+ 3 \left[ \diagra{\, \tablacsedd\,}
+ \diagra{\, \tablacseudu\,} \right] $
\\[.1cm]
\indent
$ \qquad \qquad \qquad
+  \left[ \diagra{\,\tablau\,} -1 \right]
\left[ \diagra{\,\tablath\,} +
\diagra{\, \raisebox{1.3ex}{\tablatv}\,} \right]
$\\[.1cm]
\indent
$ \qquad \qquad \qquad
+  \left[ 8\diagra{\,\tablau\,} - 18  \right]
\diagra{\, \raisebox{.65ex}{\tablatdu}\,}
+  \left[ 8\diagra{\,\tablau\,} - 10  \right]
\diagra{\, \raisebox{.65ex}{\tablatud}\,}
$\\[.1cm]
\indent
$ \qquad \qquad \qquad
+\, 3 \left[ \dbinom{\diagra{\,\tabladh\,}}{2}
+ \dbinom{\diagra{\, \raisebox{.65ex}{\tabladv}\,}}{2} \right]
+ 4 \left[ \diagra{\, \tabladh\,} \,
\diagra{\, \raisebox{.65ex}{\tabladv} \,}
- \diagra{\, \raisebox{.65ex}{\tablatud} \,} \right]
$\\[.1cm]
\indent
$ \qquad \qquad \qquad
+  \left[ 17\dbinom{\diagra{\,\tablau\,}}{2} -21 \diagra{\,\tablau\,}
+ 22 \right]
\left[ \diagra{\,\tabladh\,} +
\diagra{\, {\raisebox{.65ex}{\tabladv}} \,} \right]
$\\[.1cm]
\indent
$ \qquad \qquad \qquad
+\, 72 \dbinom{\diagra{\,\tablau\,}}{4} -18\dbinom{\diagra{\,\tablau\,}}{3}
+4\dbinom{\diagra{\,\tablau\,}}{2} -\diagra{\,\tablau\,}$.

\indent{\rm $\phantom{\rm 1}$(12)}
$\ccuad{(n-4,1^4)} =  \phantom{5} $\\
\indent
$ \qquad \qquad \qquad
\diagra{\, \tablacvdd\,}
+ \diagra{\, \tablacsedd\,}
+ \diagra{\, \tablacseudu\,} $
\\[.1cm]
\indent
$ \qquad \qquad \qquad
+  \left[ 3\diagra{\,\tablau\,} - 7  \right]
\diagra{\, \raisebox{.65ex}{\tablatdu}\,}
+  \left[ 3\diagra{\,\tablau\,} - 4  \right]
\diagra{\, \raisebox{.65ex}{\tablatud}\,}
$\\[.1cm]
\indent
$ \qquad \qquad \qquad
+ \dbinom{\diagra{\,\tabladh\,}}{2}
+ \dbinom{\diagra{\, \raisebox{.65ex}{\tabladv}\,}}{2}
+ 2 \left[ \diagra{\, \tabladh\,} \,
\diagra{\, \raisebox{.65ex}{\tabladv} \,}
- \diagra{\, \raisebox{.65ex}{\tablatud} \,} \right]
$\\[.1cm]
\indent
$ \qquad \qquad \qquad
+  \left[ 5\dbinom{\diagra{\,\tablau\,}}{2} - 6 \diagra{\,\tablau\,}
+ 6 \right]
\left[ \diagra{\,\tabladh\,} +
\diagra{\, {\raisebox{.65ex}{\tabladv}} \,} \right]
$\\[.1cm]
\indent
$ \qquad \qquad \qquad
+\, 24 \dbinom{\diagra{\,\tablau\,}}{4} -6\dbinom{\diagra{\,\tablau\,}}{3}
+2\dbinom{\diagra{\,\tablau\,}}{2} -\diagra{\,\tablau\,} +1$.

\end{enumerate}
}
\end{evakron}

\begin{obse}
{\em
Formula 2 appears already in~\cite{gara}.
Formula 3 appears in a different form in~\cite{sax}.
See also Lemmas~1-4 in Saxl's paper~\cite{sax}, where there are equivalent expressions
to our formulas~2, 3 and 5 from number~\ref{para:evaliri}.
Also formulas 1-4 here and formulas 1-4 in number~\ref{para:evaliri}
appear in Lemmas~4.1 and~4.2 in~\cite{zis}.
The graphical notation developed here permits an easier description of the formulas
when the depth of $\nu$ grows.
Formulas 6 and 7 appeared for the first time in~\cite{vpp}.
Formulas 8-12 in numbers~\ref{para:evaliri} and~\ref{para:evakron} were
calculated jointly by Avella-Alaminos and Vallejo and appear in~\cite{avala}.
}
\end{obse}

\section{The Saxl conjecture} \label{sec:saxl}

This section contains other main results of the paper.
Let $\rho_k = (k,k-1, \dots, 2,1)$ denote the \emph{staircase} partition
of size $n_k = \binom{k+1}{2}$.
The Saxl conjecture asserts that for all $k\ge 1$ the Kronecker square
$\cara{\rho_k}\otimes \cara{\rho_k}$ contains all irreducible characters
of the symmetric groups $\sime{n_k}$ as components.
See~\cite{ppv} for more information about the conjecture and some results
towards its proof.
Here we apply the results of Section~\ref{sec:evalk} to the study of Saxl's conjecture
and show what we believe to be a surprising result (Theorem~\ref{teor:kronescalerapoli}):
for each partition $\ol\nu$ of $d$ there is a piecewise polynomial function
$\fun {s_{\ol\nu}} {[0,\infty)} \real$ with the property that for all $k$ such that $(n_k-d, \ol\nu)$
is a partition one has ${\sf g}(\rho_k, \rho_k, (n_k-d, \ol\nu)) = s_{\ol\nu}(k)$.
This is the more surprising since the product $\cara{\rho_k} \otimes \cara{\rho_k}$ seems
to be the most difficult product of size $n_k$ to evaluate (see~\cite[p. 93]{lit}).
We apply this result to show (Theorem~\ref{teor:mi-saxl}) that
${\sf g}(\rho_k, \rho_k, (n_k-d, \ol\nu))$ is positive for all but at most $2d$ values of $k$.
We conclude the section with the explicit evaluation of $s_{\ol\nu}$ for all partitions of
size $|\ol\nu| \le 5$, and two conjectures, one of which is equivalent to Saxl's conjecture.

\begin{nota} \label{nota:escaleras}
{\em
For each $k\ge 1$, let $\zeta_k = b(\rho_k)$ denote the principal border strip
of $\rho_k$ (see Paragraph~\ref{para:young}).
Then $|\zeta_k|= 2k-1$.
Denote also $Z_k = \clase{\zeta_k}$.
}
\end{nota}

\begin{ejem}
{\em
For $k=4$ we have
$\rho_4 =
\begin{picture} (40,40) (-5,16)
\drawline (0,0) (10,0)
\drawline (0,10) (20,10)
\drawline (0,20) (30,20)
\drawline (0,30) (40,30)
\drawline (0,40) (40,40)
\drawline (0,0) (0,40)
\drawline (10,0) (10,40)
\drawline (20,10) (20,40)
\drawline (30,20) (30,40)
\drawline (40,30) (40,40)
\end{picture}
$
and
$\zeta_4 =
\begin{picture} (50,40) (-5,16)
\drawline (0,0) (10,0)
\drawline (0,10) (20,10)
\drawline (0,20) (30,20)
\drawline (10,30) (40,30)
\drawline (20,40) (40,40)
\drawline (0,0) (0,20)
\drawline (10,0) (10,30)
\drawline (20,10) (20,40)
\drawline (30,20) (30,40)
\drawline (40,30) (40,40)
\end{picture}
$.
}
\end{ejem}

\bigskip

\begin{defi}
{\em
By a \emph{piecewise polynomial function} (\emph{p.p.f.} for short) we mean a
continuous function $\fun f {[0,\infty)} \real$ with a finite
sequence $0 = x_0 < x_1 < \cdots < x_\ell = \infty$, $\ell\in \natural$, such that
the restriction of $f$ to each interval $[x_{i-1}, x_i]$, $i\in \nume \ell$,
is a polynomial function.
It is easy to see that linear combinations and finite products of p.p.f.'s are p.p.f.
We say that $f$ is a rational p.p.f. if the corresponding
polynomial functions defined on each interval $[x_{i-1}, x_i]$ have
rational coefficients.
}
\end{defi}

\begin{ejem} \label{ejem:ppf}
{\em
Let $a \in \natural$ and $b\in \noneg$.
Define $f_{a,b}(x) = 0$, for all $x \in [0, a + b -1]$, and
\begin{equation*}
f_{a,b}(x) = \frac{(x-b)(x-b-1) \cdots (x-b-a+1)}{a!} = \binom{x-b}{a},
\end{equation*}
for all $x \ge a + b -1$.
Then $f_{a,b}$ is a rational p.p.f.,
it is polynomial of degree $a$ in the interval $[a + b -1, \infty)$,
it is positive in the interval $(a + b -1, \infty)$ and we have $f_{a,b}(n)\in \noneg$,
for all $n\in \noneg$.
We could have defined $f_{a,b}$ to be zero in the interval $[0,b]$ and $\binom{x-b}{a}$
in the interval $[b, \infty)$, but it seemed more convenient that $f_{a,b}$ were zero
on the bigger interval.
}
\end{ejem}

\begin{mapfd} \label{para:poli-snake}
{\em
Let $D$ be a nonempty diagram class and let
$C_1^{\sqcup a_1} \sqcup \cdots \sqcup C_m^{\sqcup a_m}$ be a sorted decomposition
of $D$ (see Paragraph~\ref{para:sorted}).
We define a function $\fun {f_D}{[0, \infty)}\real$ as follows:
if there is some $i\in \nume m$ such that the principal border strip
(see Paragraph~\ref{para:clase-diag}) $B(C_i) \neq Z_n$ for all $n \in \natural$,
we define $f_D$ to be the zero map.
Let $k\in \natural$.
In this case $\diagralibre{\rho_k}{C_i} = 0$.
Therefore $\diagralibre{\rho_k}D = 0$, in other words, $\diagralibre{\rho_k}D = f_D(k)$.

If, for each $l\in \nume m$, there is some $n_l \in \natural$ such that
$B(C_l) = Z_{n_l}$, we let, for each $0 \le i \le m$,
\begin{equation*}
b_i = \sum_{t \in \nume i} (n_t-1)a_t + \sum_{t > i} n_t a_t,
\end{equation*}
and define $f_D = f_{a_1,b_1} \cdots f_{a_m,b_m}$.
Note that $b_0$ is the number $r_D$ of removable squares of $D$
(see Paragraph~\ref{nota:remueve-sesgado}),
$a_i + b_i = b_{i-1}$, for all $i \in \nume m$ and that
$b_0 > b_1 > \cdots > b_m \ge 0$.
In this case $f_D$ is a rational p.p.f.,
it is the zero map in the interval $[0, b_0 -1]$,
it is polynomial of degree $a_1 + \cdots + a_m$ in the interval $[b_0 -1, \infty)$,
and it is positive in the interval $(b_0 -1, \infty)$.
Finally, we define $\fun {f_\vacio} {[0,\infty)} \real$ to be the constant map
equal to 1.
}
\end{mapfd}

\begin{lema} \label{lema:remueve-snakes}
Let $D = C_1^{\sqcup a_1} \sqcup \cdots \sqcup C_m^{\sqcup a_m}$
be a sorted decomposition.
Suppose that for each $i \in \nume m$ there is some $n_i$ such that $B(C_i)= Z_{n_i}$.
Let $k\in \natural$ and denote $a_0 = k - \sum_{i\in \nume m} n_i a_i$.
If $a_0 <0$, then $\diagralibre{\rho_k}{D} = 0$.
If $a_0 \ge 0$, then
\begin{equation*}
\diagralibre{\rho_k}{D} =
\binom{a_0 + a_1 + \cdots + a_m}{a_0, a_1, \dots, a_m}.
\end{equation*}
\end{lema}
\begin{proof}
Recall that the sum $\sum_{i\in \nume m} n_i a_i$ is the number of removable squares of $D$.
In case $a_0 < 0$, $D$ does not fit in $\rho_k$.
Therefore $\diagralibre{\rho_k}{D} = 0$.
Let us assume that $a_0\ge 0$, then our hypothesis on the decomposition
of $D$ implies $\diagralibre{\rho_k}{D} > 0$.
For each $\rho_k$-removable diagram $\sigma\in D$, we define a word $w_\sigma$
of length $a_0 + a_1 + \cdots + a_m$ in the alphabet $0$, $1, \dots, m$
such that $i$ appears $a_i$ times, for $0 \le i \le m$.
We do this by looking in which order (say from top to bottom) each
component of $\sigma$ appears in $\zeta_k$.
A $0$ in a $w_\sigma$ stands for an empty removable square in $\zeta_k$, that is,
a square of $\zeta_k$ not occupied by $\sigma$; a positive
$i$ in $w_\sigma$ stands for a component of $\sigma$ that belongs to $C_i$.
This establishes a one-to-one correspondence between ${\sf R}_{\rho_k}(D)$ and
the set of words just described.
Since the multinomial coefficient counts such set of words, the lemma follows.
\end{proof}

\begin{prop} \label{prop:polisnakes}
Let $D = C_1^{\sqcup a_1} \sqcup \cdots \sqcup C_m^{\sqcup a_m}$ be a sorted
decomposition.
Then for all $k \in \natural$ we have
\begin{equation*}
\diagralibre{\rho_k}{D} = f_D(k).
\end{equation*}
\end{prop}
\begin{proof}
If there is some $i\in \nume m$ such that $B(C_i) \neq Z_n$ for all $n \in \natural$,
the claim follows from Paragraph~\ref{para:poli-snake}.
Suppose then, that for each $i\in \nume m$, there is some $n_i \in \natural$ such that
$B(C_i) = Z_{n_i}$.
Let $b_0, \dots, b_m$ be defined as in Paragraph~\ref{para:poli-snake}.
Then, by Lemma~\ref{lema:remueve-snakes}, we have, for any $k \ge b_0$ that
\begin{equation*}
\diagralibre{\rho_k}{D} =
\binom{k - b_0 + a_1 + \cdots + a_m}{k - b_0, a_1, \dots, a_m}.
\end{equation*}
Since
\begin{equation*}
\binom{k - b_0 + a_1 + \cdots + a_m}{k - b_0, a_1, \dots, a_m} =
\binom{k - b_0 + a_1 + \cdots + a_m}{a_m} \cdots \binom{k - b_0 + a_1}{a_1},
\end{equation*}
and since $k -b_i = k - b_0 + a_1 + \cdots + a_i$, we get $\diagralibre{\rho_k}{D} = f_D(k)$.
Now, if  $k < b_0$, since $b_0$ is the number of removable squares of $D$,
we have $\diagralibre{\rho_k}{D} = 0$, and, by definiton, $f_D(k) = 0$.
The proof is complete.
\end{proof}

\begin{obse}
{\em
Note that $\diagralibre{\rho_k}{Z_m} = k-m+1$, for all $k \ge m-1$ and that
for $k \le m-1$ one has $\diagralibre{\rho_k}{Z_m} = 0$.
That is why we have to consider piecewise polynomial functions.
}
\end{obse}

\begin{nota} \label{nota:cota-infe}
{\em
In the next theorem we need the following notation.
Let $\ol\nu = (\nu_2, \dots, \nu_r)$ be a partition of $d$.
For each $k \in \natural$, let $n_k = |\rho_k| = \binom{k+1}{2}$
and $t(\ol\nu)$ be the smallest integer that is greater or equal to
$\frac{-1 + \sqrt{1 + 8(\nu_2 + d)}}{2}$.
Hence $n_k \ge d + \nu_2$ if and only if  $k \ge t(\ol\nu)$.
In other words, $(n_k - d, \ol\nu )$ is a partition of $n_k$ if and only if
$k \ge t(\ol\nu)$.
Finally, let $c({\ol\nu}) = \max \{ d-1, t(\ol\nu)\}$.
}
\end{nota}

\begin{teor} \label{teor:kronescalerapoli}
Let $\ol\nu$ be a partition of $d$.
Then there is a rational piecewise polynomial function
\begin{equation*}
\fun {s_{\ol\nu}} {[0,\infty)} \real
\end{equation*}
such that for all $k \ge t(\ol\nu) $
\begin{equation*}
{\sf g} (\rho_k, \rho_k, ( n_k - d, \ol\nu ) ) = s_{\ol\nu}(k).
\end{equation*}
Moreover, $s_{\ol\nu}$ is a polynomial function of degree $d$ in the interval
$[ c(\ol\nu), \infty )$ with leading coefficient $f^{\ol\nu}$.
\end{teor}
\begin{proof}
Define, for each $x\in [0,\infty)$,
\begin{equation} \label{ecua:defi-poli-bueno}
s_{\ol\nu}(x) =
\sum_{  D, \, |D| \le d}\ \sum_{ \substack{T\in \tirabo {\wt\nu} \\ e(T) = |D|}}
\signo T\, {\sf lr}(D,D; \ol\tau(T))\, f_D(x).
\end{equation}
Thus $s_{\ol\nu}$ is a rational p.p.f.
By Proposition~\ref{prop:polisnakes}, we have, for all $k\in \natural$, that
\begin{equation*}
s_{\ol\nu}(k) =
\sum_{D, \, |D| \le d}\ \sum_{ \substack{T\in \tirabo {\wt\nu} \\ e(T) = |D|}}
\signo T\, {\sf lr}(D,D; \ol\tau(T))\, \diagralibre {\rho_k} D,
\end{equation*}
and, by Theorem~\ref{teor:kron-combin}, we have, for all $k \ge t(\ol\nu)$, that
\begin{equation*}
s_{\ol\nu}(k) = {\sf g} (\rho_k, \rho_k, ( n_k - d, \ol\nu ) ).
\end{equation*}
Let $D$ be a nonempty diagram class that corresponds to a summand of $s_{\ol\nu}(x)$ with
$f_D \neq 0$.
Let $D = C_1^{\sqcup a_1} \sqcup \cdots \sqcup C_m^{\sqcup a_m}$ be a sorted decomposition.
Since $f_D \neq 0$, there is some $k\in \natural$ such that $\diagralibre {\rho_k}D >0$.
Then, for each $i\in \nume m$, there is some $n_i \in \natural$
with $B(C_i) = Z_{n_i}$.
Recall that $b_0 = \sum_{i\in \nume m} n_i a_i$ is the number of removable squares of $D$.
Observe that $b_0 \le |D| \le d$ and $b_0 = d$ if and only if $D = \Delta_d$.
Since the degree of $f_D$ in the interval $[b_0 -1, \infty)$ is $a_1 + \cdots + a_m$, the degree
of $f_D$ is at most $d$ and the degree of $f_D$ is $d$ if and only if $D=\Delta_d$.
From the previous discussion we conclude that in the interval $[d-1, \infty)$ the only
summand of degree $d$ in equation~\eqref{ecua:defi-poli-bueno} is $f_{\Delta_d}$.
All other non-zero summands are polynomial of degree smaller that $d$ in the same
interval.
Therefore $s_{\ol\nu}$ is polynomial of degree $d$ in $[ c(\ol\nu), \infty )$.

It remains to show that the coefficient of $x^d$ in $s_{\ol\nu}(x)$
in the interval $[ c(\ol\nu), \infty )$ is $f^{\ol\nu}$.
By Proposition~\ref{prop:coefi-rldeltakdos}.1 the coefficient of $\diagralibre{\rho_k}{\Delta_d}$
in equation~\eqref{ecua:kron-merabuena} is $f^{\ol\nu} d!$
Therefore, the coefficient of $x^d$ in $s_{\ol\nu}(x)$ is the coefficient of $x^d$ in
$f^{\ol\nu} d! f_{\Delta_d}(x)$, which is the coefficient of $x^d$ in
$f^{\ol\nu} d! \binom{x}{d}$.
The proof of the theorem is complete.
\end{proof}

\begin{lema} \label{lema:coef-dmu}
Let $\ol\nu = (m, 1^{d-m})$ be a partition of $d$.
Then the coefficient of $x^{d-1}$ in $s_{\ol\nu}(x)$ in the interval $[ c(\ol\nu), \infty )$ is
$-f^{\ol\nu}\left[ \binom{d}{2} + 1 \right]$.
\end{lema}
\begin{proof}
By a similar analysis to the one given in Theorem~\ref{teor:kronescalerapoli}
we see that the only contributions to the coefficient of $x^{d-1}$ come from the summands
in equation~\eqref{ecua:defi-poli-bueno} corresponding to $\Delta_d$ and $\Delta_{d-1}$.
By Propositions~\ref{prop:coefi-rldeltakdos}.1 and~\ref{prop:polisnakes} the coefficient of
$f_{\Delta_d}$ in the interval $[ d-1, \infty )$ is $f^{\ol\nu} d!$ and by
Propositions~\ref{prop:coefi-rldeltakdos}.2 and~\ref{prop:polisnakes} the coefficient of
$f_{\Delta_{d-1}}$ in the interval $[ d-2, \infty )$ is $-(d-1)!\binom{d-1}{m-1}$.
Recall $f_{\Delta_k}(x) = \binom{x}{k}$.
Therefore, the coefficient of $x^{d-1}$ in $s_{\ol\nu}(x)$ in the interval $[ c(\ol\nu), \infty )$ is
$-f^{\ol\nu} \binom{d}{2} - \binom{d-1}{m-1}$.
Since the hook length formula implies that $f^{\ol\nu} = \binom{d-1}{m-1}$,
the lemma follows.
\end{proof}

\begin{prob}
{\em
Is it true that for any partition $\ol\nu$ of $d$ the coefficient of $x^{d-1}$ in
$s_{\ol\nu}(x)$ in the interval $[ c(\ol\nu), \infty )$ is
$-f^{\ol\nu}\left[ \binom{d}{2} + 1 \right]$?
According to Table~\ref{tabla:polyfun} this is true for $d\le 5$.
}
\end{prob}

\begin{polifun} \label{para:polifun}
{\em
Let $\ol\nu \vdash d$.
We have seen that $s_{\ol\nu}$ is a polynomial function of degree $d$ in the
$[c({\ol\nu}), \infty)$.
We call this the \emph{main interval} of $s_{\ol\nu}$.
The only partitions $\ol\nu$ for which $t(\ol\nu) \ge d-1$ are all partitions of
$1$, $2$, $3$, $4$, all partitions of 5 but $(1^5)$, and the partitions $(6)$ and $(5,1)$.
For all these partitions $s_{\ol\nu}$ is a polynomial function.
For the remaining partitions the main interval is not the whole domain of $s_{\ol\nu}$.
}
\end{polifun}

So, we obtain the following approximation to the Saxl conjecture:

\begin{teor} \label{teor:mi-saxl}
Let $\ol\nu$ be a partition of $d$.
Then ${\sf g} (\rho_k, \rho_k, ( n_k - d, \ol\nu ) ) >0$, for all
$k \ge t(\ol\nu)$, with the possible exception of at most
$2d - \frac{1 + \sqrt{1 + 8(\nu_2 + d)}}{2}$ $k$'s.
\end{teor}
\begin{proof}
Since $s_{\ol\nu}$ is a polynomial function of degree $d$ in the interval
$[ c(\ol\nu), \infty )$, the biggest number of integer zeros of $s_{\ol\nu}$
in the interval $[ t(\ol\nu), \infty )$ would be attained when $d-1 > t(\ol\nu)$,
$s_{\ol\nu}(k) = 0$ for all $k\in \natural \cap [ t(\ol\nu), d-1)$ and
$s_{\ol\nu}$ has $d$ integer zeros in $[ d-1, \infty )$.
This number would be $d - 1 - t(\ol\nu) + d$.
The claim follows.
\end{proof}

\begin{ejem} \label{ejem:polikorn}
{\em
Let $\ol\nu$ be a partition of size at most $5$.
In Table~\ref{tabla:polyfun} we list the polynomial functions $s_{\ol\nu}$ in the
interval $[ c(\ol\nu), \infty )$, its real roots in the same interval (approximated
only to two decimals) and $t(\ol\nu)$.
Observe that all real roots of $s_{\ol\nu}$ are located in the interval $[0, t(\ol\nu) )$.
We used Sage~\cite{sage} in these computations.
If $\ol\nu \neq (1^5)$, then $t(\ol\nu) = c(\ol\nu)$.
In these cases $s_{\ol\nu}$ is a polynomial function in the interval
$[ t(\ol\nu), \infty )$.
Note that $t(1^5) = 3$ and $4 = c(1^5)$.
Therefore $s_{(1^5)}$ is defined by different polynomials in the intervals
$[3,4]$ and $[4, \infty)$.
All the nonzero summands $f_D$ of $s_{(1^5)}$ are polynomial in the interval
$[3, \infty)$ with the exception of $f_{\Delta_5}$, which,
by Propositions~\ref{prop:coefi-rldeltakdos}.1 and~\ref{prop:polisnakes},
has coefficient 120.
This map is polynomial in $[4,\infty)$ and 0 in $[3,4]$.
Since $f_{\Delta_5}(3) = 0 = \binom{3}{5}$, we can substitute $f_{\Delta_5}$
by $120 \binom{x}{5}$.
In this way the polynomial $p(x) $ appearing in the last row of Table~\ref{tabla:polyfun}
satisfies $p(k) = {\sf g}(\rho_k, \rho_k, (|\rho_k| - 5, 1^5))$ for all $k \ge 3$.
}
\end{ejem}

\begin{table}[h]
\begin{center}
\begin{tabular}{ | r || c | c | c |} \hline
$\ol\nu$ & $s_{\ol\nu}$ in the interval $[ c(\ol\nu), \infty )$ & real roots of $s_{\ol\nu}$
& $t(\ol\nu)$ \\ \hline \hline
$(1)      $ & $ x-1 $ & $1$ & $2$ \\ \hline
$(2)      $ & $ x^2 -2x $     & $0$, $2$ & $3$ \\
$(1^2)    $ & $ x^2 -2x + 1 $ & $1$, $1$ & $2$ \\ \hline
$(3)      $ & $ x^3 -4x^2 + 4x -1 $ & $0.38$, $1$, $2.62$ & $3$ \\
$(2,1)    $ & $ 2x^3 -8x^2 +8x -1 $ & $0.15$, $1.4$, $2.45$ & $3$ \\
$(1^3)    $ & $ x^3 -4x^2 + 5x -2 $ & $1$, $1$, $2$ & $3$ \\ \hline
$(4)      $ & $ x^4 - 7x^3 + 17x^2 - 18x + 7 $    & $1$, $3.32$ & $4$ \\
$(3,1)    $ & $ 3x^4 - 21x^3 + 51x^2 - 51x + 18 $ & $1$, $1$, $2$, $3$ & $4$ \\
$(2^2)    $ & $ 2x^4 - 14x^3 + 34x^2 - 33x + 11 $ & $0.81$, $1$ & $3$ \\
$(2,1^2)  $ & $ 3x^4 - 21x^3 + 52x^2 - 53x + 18 $ & $0.69$, $1.63$, $2$, $2.68$ & $3$ \\
$(1^4)    $ & $ x^4 - 7x^3 + 18x^2 - 20x + 8    $ & $1$, $2$, $2$, $2$ & $3$ \\ \hline
$(5)      $ & $ x^5 - 11x^4 + 48x^3 - 106x^2 + 119x - 54 $ & $1.56$, $2$, $3.79$ & $4$\\
$(4,1)    $ & $ 4x^5 - 44x^4 + 192x^3 - 420x^2 + 462x - 203 $ & $1.41$, $2.3$, $3.52$ & $4$ \\
$(3,2)    $ & $ 5x^5 - 55x^4 + 240x^3 - 522x^2 + 567x - 245 $ & $1.42$, $2.46$, $3.27$ & $4$ \\
$(3,1^2)  $ & $ 6x^5 - 66x^4 + 289x^3 - 632x^2 + 690x - 300 $ & $1.54$, $2$, $3.25$ & $4$ \\
$(2^2,1)  $ & $ 5x^5 - 55x^4 + 241x^3 - 526x^2 + 571x - 246 $ & $1.53$, $2$, $3$ & $4$ \\
$(2,1^3)  $ & $ 4x^5 - 44x^4 + 194x^3 - 428x^2 - 470x - 204 $ & $1.41$, $2$, $3$ & $4$ \\
$(1^5)    $ & $ x^5 - 11x^4 + 49x^3 -110x^2 + 124x -56 $      & $2$, $2$, $2$ & $3$ \\ \hline
\end{tabular}
\caption{First polynomial functions}
\label{tabla:polyfun}
\end{center}
\end{table}

From the information in Table~\ref{tabla:polyfun} we get

\begin{coro}
{\em
Let $d \in \nume 5$ and $\ol\nu \vdash d$.
Then for all $k \in [t(\ol\nu), \infty )$ the character $\cara{(n_k - d, \ol\nu)}$ is a
component of $\cara{\rho_k} \otimes \cara{\rho_k}$.
}
\end{coro}

We propose two conjectures concerning the maps $s_{\ol\nu}$.
The first is equivalent to Saxl's conjecture, the second is about the coefficients
of $s_{\ol\nu}$.

\begin{conje}
{\em
Let $\ol\nu \vdash d$.
Then $s_{\ol\nu}$ is positive in the interval $[ t(\ol\nu), \infty )$.
}
\end{conje}

\begin{conje}
{\em
Let $\ol\nu \vdash d$.
Then the polynomial map $s_{\ol\nu}$ has nonzero integer coefficients in the interval
$[c(\ol\nu), \infty )$ and its signs alternate.
The sign of the coefficient of $x^k$ in $s_{\ol\nu}(x)$ is $(-1)^{d-k}$ for all
$0 \le k \le d$.
}
\end{conje}

\section{Extension of the diagrammatic method to arbitrary Kronecker coefficients} \label{sec:ext}

In this section we show how our diagrammatic method extends to arbitrary Kronecker coefficients.
All results follow straightforwardly from what has been done in previous sections.

\begin{defi}
{\em
Let $\la$, $\mu$ be partitions of the same number and let $D$, $E$ be diagram classes of
the same size.
Let $\rho$ be a $\la$-removable diagram, and denote by $\la\setminus \rho$ the set-theoretic
complement of $\rho$ in $\la$.
Thus $\la\setminus\rho$ is a partition.
We define
\begin{equation*}
\Diagrados DE = \{ (\rho, \sigma) \in \Diagra D \times \Diagralibre \mu E \mid
\la \setminus \rho = \mu \setminus \sigma \}
\end{equation*}
and
\begin{equation*}
\diagrados DE = \# \Diagrados DE.
\end{equation*}
}
\end{defi}

Proposition~\ref{prop:multiliridiag} can be extended in the following way.
The proof of the new result is essentially the same.

\begin{prop} \label{prop:multilirigen}
Let $\la$, $\mu$ be partitions of $n$, $\pi$ a composition of $n$ and $d = |\ol\pi|$.
Then
\begin{equation*}
\multiliri \pi = \sum_{D,\, E \in \eseefe Dd} {\sf lr}(D,E;\ol\pi)\, \diagrados DE.
\end{equation*}
\end{prop}

Combining Propositions~\ref{prop:robtau} and~\ref{prop:multilirigen} we obtain
a generalization of Theorem~\ref{teor:kron-combin}, that is,
an enhancement of the RT method for arbitrary Kronecker coefficients.

\begin{teor} \label{teor:kron-combin-dos}
Let $\ol\nu = (\nu_2, \dots, \nu_r)$ be a partition of $d$ and
$n\ge d + \nu_2$.
Then, for any partitions $\la$, $\mu$ of $n$, we have
\begin{equation} \label{ecua:kron-merabuenados}
{\sf g}(\la,\mu, (n-d,\ol\nu)) = \sum_{k=0}^d \sum_{D,\, E \in \eseefe Dk}
\sum_{ \substack{T\in \tirabo {\wt\nu} \\ e(T) = k}} \signo T\, {\sf lr}(D,E; \ol\tau(T))\, \diagrados DE.
\end{equation}
\end{teor}

Note that when $\la = \mu$, equation~\eqref{ecua:kron-merabuenados} reduces to
equation~\eqref{ecua:kron-merabuena}.

\begin{probs}
{\em
Is there an analogous result to Theorem~\ref{teor:elbueno} for $\diagrados DE$?
Are there analogous results to Theorems~\ref{teor:polikrontirabo} and~\ref{teor:tiras-borde-kron}
for ${\sf g}(\la,\mu, (n-d,\ol\nu))$?
}
\end{probs}

\section{A stability property for Kronecker coefficients} \label{sec:estab}

In this section we prove a new property for Kronecker coefficients.
It generalizes the stability property noted by Murnaghan~\cite{mur} and proved since
in different ways~\cite{bri,lit2,thi,vejc}.
In particular, our graphical approach yields a new proof of Murnaghan's original
stability property.
It might be possible to find a proof of this new property using other techniques,
but it was the diagrammatic method exposed here that permitted its discovery.

\begin{nota}
{\em
For any partition $\mu = \vector \mu q$, given $k\in \natural$ and $i\in \nume q$, denote
\begin{equation*}
\mu^{(i,k)} = (\mu_1 + k, \dots, \mu_i +k, \mu_{i+1}, \dots, \mu_q).
\end{equation*}
}
\end{nota}

\begin{teor} \label{teor:estab-nuevo}
Let $\la$, $\mu$, $\nu$ be partitions of $n$ and
$i \le \min\{\ell(\la), \ell(\mu)\}$.
If $\la_i - \la_{i+1} \ge \prof\nu$ and $\mu_i - \mu_{i+1} \ge \prof\nu$, then
for all $k\in \natural$  we have
\begin{equation} \label{ecua:estab-nuevo}
{\sf g}(\la^{(i,k)}, \mu^{(i,k)}, \nu^{(1,ki)}) = \coefi.
\end{equation}
\end{teor}
\begin{proof}
Let $d = \prof\nu$, $\nu = (n-d, \ol\nu)$ and let $D$, $E$ be diagram classes of size
at most $d$.
Since $\la_i - \la_{i+1} \ge d$, the correspondence $\la/\alpha \asocia \la^{(i,k)}/\alpha^{(i,k)}$
defines a bijection between $\Diagra D$ and $\Diagralibre {\la^{(i,k)}} D$.
Thus $\diagralibre {\la^{(i,k)}} D= \diagra D$.
Similarly, $\diagralibre {\mu^{(i,k)}} E = \diagralibre \mu E$.
Therefore $\diagrados DE = {\sf r}_{\la^{(i,k)},\,\mu^{(i,k)}}(D,E)$.
And since $\nu^{(1,ki)}$ and $\nu$ differ only in their first part,
Theorem~\ref{teor:kron-combin-dos} implies that both coefficients in
equation~\eqref{ecua:estab-nuevo} are equal.
\end{proof}

\begin{obses}
{\em
(1) The case $i=1$ yields Murnaghan's stability property.

(2) Due to the symmetry ${\sf g}(\la^\prime, \mu^\prime, \nu) =
\coefi$, one also has a similar stability theorem for columns.

(3) Formulas~~\eqref{ecua:kron-merabuena} and~\eqref{ecua:kron-merabuenados} should be
helpful to improve the known bounds of stability for a given Kronecker coefficient,
see~\cite{bor,bri,vejc}.
}
\end{obses}

\begin{comen}
{\em
Following the suggestion of a reviewer, I sketched, in this new version,
how to extend the diagrammatic method from Kronecker squares to
arbitrary Kronecker products (see Section~\ref{sec:ext}) and noted that the diagrammatic
method permitted also to extend the new stability property from Kronecker squares to
arbitrary Kronecker coefficients.
So, Theorem~9.2 in~\cite{valfpsac}, became Theorem~\ref{teor:estab-nuevo} here.
After the revised version of this paper was submitted for publication
I learned that Igor Pak and Greta Panova had a similar result to our
Theorem~\ref{teor:estab-nuevo} (see~\cite[Theorem~4.1]{pp3}).
Their hypothesis for stability is, in the notation of Theorem~\ref{teor:estab-nuevo},
\begin{equation*}
\min\{\la_i,\mu_i\} - \max\{\la_{i+1},\mu_{i+1}\} \ge \prof \nu,
\end{equation*}
which is weaker than ours.
On the other hand they prove something more (Theorem~1.1), namely, that, as a function of $k$,
the coefficient of ${\sf g}(\la^{(i,k)}, \mu^{(i,k)}, \nu^{(1,ki)})$ is monotone increasing.
}
\end{comen}

\vskip 1.5pc

\noindent \textbf{Acknowledgements.}
I am grateful to Christine Bessenrodt, Igor Pak and Miguel Raggi for helpful
suggestions.
I also would like to thank the referees for their very careful reading and their
helpful suggestions.
This paper was supported by UNAM-DGAPA IN102611 and IN108314.

\end{document}